\documentclass[12pt,twoside]{amsart}
\usepackage{mathrsfs}
\usepackage{enumitem}
\usepackage{amsmath, amsfonts, amssymb, amsthm, amscd}
\usepackage[all]{xy}
\usepackage{fancyhdr}
\usepackage{hyperref}
\usepackage{geometry}
\usepackage{makecell}  
\usepackage{lipsum}

\usepackage[mathscr]{eucal}

\theoremstyle{plain}
\newtheorem{thm}{Theorem}[section]

\newtheorem{theorem}[thm]{Theorem}
\newtheorem{lemma}[thm]{Lemma}
\newtheorem{corollary}[thm]{Corollary}
\newtheorem{proposition}[thm]{Proposition}

\newtheorem{definition}[thm]{Definition}

\theoremstyle{remark}

\newtheorem{remark}[thm]{Remark}

\newtheorem{defn-thm}[thm]{Definition-Theorem}
\newtheorem{defn-lem}[thm]{Definition-Lemma}

\renewcommand{\bar}{\overline}

\renewcommand{\phi}{\varphi}

\newcommand{\C}{{\mathbb C}}

\newcommand{\R}{{\mathbb R}}

\newcommand{\Q}{{\mathbb Q}}

\newcommand{\M}{{\mathcal M}}
\newcommand{\T}{{\mathcal T}}

\renewcommand{\tilde}{\widetilde}
\newcommand{\p}{{\Phi}}

\begin{document}

\def\dW{\mbox{diff\:}\times \mbox{Weyl\:}}
\def\End{\operatorname{End}}
\def\Hom{\operatorname{Hom}}
\def\Aut{\operatorname{Aut}}
\def\Diff{\operatorname{Diff}}
\def\im{\operatorname{im}}
\def\tr{\operatorname{tr}}
\def\Pr{\operatorname{Pr}}
\def\Z{\bf Z}
\def\O{\mathcal{O}}
\def\CP{\mathbb{C}\mathbb{P}}
\def\P{\Phi}
\def\TT{\mathcal {T}_m^H}

\def\Q{\bf Q}
\def\R{\bf R}
\def\C{\mathbb{C}}
\def\H{H_{\mathrm{pr}}}
\def\Hil{\mathcal{H}}
\def\proj{\operatorname{proj}}
\def\id{\mbox{id\:}}
\def\a{\mathfrak a}
\def\d{\partial}
\def\tO{\tilde{\Omega}}

\def\b{\beta}
\def\c{\gamma}
\def\p{\partial}
\def\f{\frac}
\def\i{\sqrt{-1}}
\def\t{\tau}
\def\T{\mathcal{T}}
\def\Tan{\mathrm{T}^{1,0}}
\def\aTan{\mathrm{T}^{0,1}}
\def\Kahler{K\"{a}hler\:}
\def\w{\omega}
\def\X{\mathcal{X}}
\def\K{\mathcal {K}}
\def\m{\mu}
\def\M{\mathcal {M}}
\def\Z{\mathcal {Z}_m}
\def\ZZ{\mathcal {Z}_m^H}

\newcommand{\bp}{\bar{\partial}}

\def\v{\nu}
\def\D{\mathcal{D}}
\def\U{\mathcal {U}}
\def\V{\mathcal {V}}
\def\Omegak{\frac{1}{k!}\bigwedge\limits^k\mu\lrcorner\Omega}
\def\Omegakp{\frac{1}{(k+1)!}\bigwedge\limits^{k+1}\mu\lrcorner\Omega}
\def\Omegakpp{\frac{1}{(k+2)!}\bigwedge\limits^{k+2}\mu\lrcorner\Omega}
\def\Omegakm{\frac{1}{(k-1)!}\bigwedge\limits^{k-1}\mu\lrcorner\Omega}
\def\Omegakmm{\frac{1}{(k-2)!}\bigwedge\limits^{k-2}\mu\lrcorner\Omega}
\def\Omegakk{\Omega_{i_1,i_2,\cdots,i_k}}
\def\Omegakkp{\Omega_{i_1,i_2,\cdots,i_{k+1}}}
\def\Omegakkpp{\Omega_{i_1,i_2,\cdots,i_{k+2}}}
\def\Omegakkm{\Omega_{i_1,i_2,\cdots,i_{k-1}}}
\def\Omegakkmm{\Omega_{i_1,i_2,\cdots,i_{k-2}}}
\def\mukm{\frac{1}{(k-1)!}\bigwedge\limits^{k-1}\mu}
\def\sumk{\sum\limits_{i_1<i_2<,\cdots,<i_k}}
\def\sumkm{\sum\limits_{i_1<i_2<,\cdots,<i_{k-1}}}
\def\sumkmm{\sum\limits_{i_1<i_2<,\cdots,<i_{k-2}}}
\def\sumkp{\sum\limits_{i_1<i_2<,\cdots,<i_{k+1}}}
\def\sumkpp{\sum\limits_{i_1<i_2<,\cdots,<i_{k+2}}}
\def\Omegakb{\Omega_{i_1,\cdots,\bar{i}_t,\cdots,i_k}}
\def\Omegakmb{\Omega_{i_1,\cdots,\bar{i}_t,\cdots,i_{k-1}}}
\def\Omegakpb{\Omega_{i_1,\cdots,\bar{i}_t,\cdots,i_{k+1}}}
\def\Omegakt{\Omega_{i_1,\cdots,\tilde{i}_t,\cdots,i_k}}

\title{Sections of Hodge bundles II: Deformation of $(p,p)$-classes and applications to K\"ahler geometry}
\author{Kefeng Liu}
\address{Mathematical Sciences Research Center, Chongqing University of Technology, Chongqing 400054, China; \newline
Department of Mathematics,University of California at Los Angeles, Los Angeles, CA 90095-1555, USA}
\email{liu@math.ucla.edu}

\author{Yang Shen}
\address{Mathematical Sciences Research Center, Chongqing University of Technology, Chongqing 400054, China}
\email{syliuguang2007@163.com}
\date{}

\vspace{-20pt}
\begin{abstract}
Let $(X,\omega_0)$ be a compact K\"ahler manifold and $\mathcal X\to B$ its Kuranishi family, where $B$ may be singular and $\dim_{\C}B\ge1$. Using explicit sections of Hodge bundles, we define an intrinsic period map and a Hodge map parametrizing nearby $(p,p)$-classes.

For deformations over irreducible analytic bases, we introduce two flat extensions of K\"ahler cones defined by the reference and moving Hodge connections. The extension associated with the reference connection admits explicit positive representatives and yields uniform upper semicontinuity, while that associated with the moving connection identifies the K\"ahler cones away from a countable union of proper analytic subsets and admits an explicit expression in terms of the period map and the Beltrami differential. These constructions provide a description of K\"ahler cones through analytic cycles and yield both local and large-scale K\"ahler stability without assuming unobstructedness.

As further applications, we generalize Green's density criterion to strong algebraic approximation and to the approximation of real $(p,p)$-forms. We also obtain an intrinsic analytic description of Hodge loci, leading to a Beltrami-differential criterion for the variational Hodge conjecture.
\end{abstract}
%
%
\maketitle
\parskip=5pt
\baselineskip=15pt


\tableofcontents



\setcounter{section}{-1}
\section{Introduction}
The present paper continues our earlier work \cite{LS26-1}, where sections of Hodge bundles were constructed by combining algebraic input from period matrices with geometric input from deformation theory via Beltrami differentials, and were used to study the global geometry of period maps for polarized manifolds.
Here we develop a non-polarized counterpart of this approach for families of compact K\"ahler manifolds, using such sections of the corresponding Hodge bundles to define an intrinsic period map and a Hodge map parametrizing nearby $(p,p)$-classes.
Our main results show that the geometry of nearby fibers---including the structures of the $\nabla_{t_0}^{1,1}$- and $\nabla^{1,1}$-flat extensions of K\"ahler cones and the behavior of real $(p,p)$-classes---is governed explicitly by the associated Beltrami differentials through Hodge-theoretic extensions of cohomology classes.

\emph{In particular, these Hodge-theoretic and positivity properties are determined entirely by the Hodge theory of the central fiber together with the corresponding Beltrami differentials, and this control persists over large regions of the base determined by a uniform bound on the operator norm of the Beltrami differential.}

These general principles have several concrete consequences.
We introduce two natural flat extensions of K\"ahler cones. The $\nabla_{t_0}^{1,1}$-flat extension yields a uniform upper semicontinuity property for K\"ahler cones and leads to a large-scale form of K\"ahler stability, while the $\nabla^{1,1}$-flat extension admits an explicit expression in terms of the period map and identifies the K\"ahler cones away from a countable union of proper analytic subsets. As additional applications, we obtain a generalization of Green's density criterion for strong algebraic approximation and for the approximation of real $(p,p)$-forms, together with an intrinsic analytic description of Hodge loci that yields a criterion for the variational Hodge conjecture.

\noindent \textbf{Intrinsic period map and Hodge map}

We begin with a compact K\"ahler manifold $(X,\omega_{0})$ and its Kuranishi family $\mathcal{X} \to B$ over an analytic space $B$ (possibly singular) contained in a polydisk $\Delta$ with $0 \in \Delta$, such that $X \cong X_{0}$. The base $B$ is defined in terms of a Beltrami differential
$$
\phi(t) = \sum_{1\le i\le N} \theta_{i} t_{i} + \frac{1}{2} \bar{\partial}^{*} G [\phi(t),\phi(t)] \in A^{0,1}(X,\Tan X),
$$
as the zero set of the analytic functions $\mathbb{H}[\phi(t),\phi(t)]$ on $\Delta$, where $\{\theta_{i}\}_{i=1}^{N}$ is a basis of $\mathbb{H}^{0,1}(X,\Tan X)$. Here $\bar{\partial}^{*}$, the Green operator $G$, and the harmonic projection $\mathbb{H}$ are the natural maps associated to the elliptic operator
$$
\triangle_{\bar \partial}:\, A^{0,*}(X,\Tan X) \to A^{0,*}(X,\Tan X),
$$
defined with respect to the initial K\"ahler form $\omega_{0}$.

In this paper, we always assume that $\dim B\ge 1$.

Let $$\eta= \{\eta_{(0)}^T, \cdots, \eta_{(n)}^T\}^T=\{[\tilde\eta_{(0)}]^T, \cdots, [\tilde\eta_{(n)}]^T\}^T$$
be the adapted basis $\eta$ with harmonic representatives $\tilde \eta$ the Hodge structure (not necessarily polarized) of weight $n$, $$H^{n}(X,\C)=\bigoplus_{p+q=n}H^{p,q}(X),$$
where each $\eta_{(p)}$, viewed as a column vector with entries in $H$, is a basis of $ H^{n-p,p}(X)$, $0 \leq p \leq n$.

We then can define the local period map 
\begin{equation}\label{intr local pm}
\Phi:\, B\to N_{-}\cap D,\, t\mapsto \Phi(t)
\in N_{-}
\end{equation}
by letting 
$$\Phi(t)=\left(
\begin{array}{cccc}
I & \Phi^{(0,1)}(t) & \cdots& \Phi^{(0,n)}(t) \\
O & I& \cdots& \Phi^{(1,n)}(t) \\
\vdots &  \vdots& \ddots&\vdots \\
O&O&\cdots&I
\end{array}\right)$$
with the non-trivial blocks determined by
\begin{equation}\label{period matrices defn}
\Phi^{(i,j)}(t)=\left(\mathbb H \left(i_{\phi(t)}^{j-i}\left((I+Ti_{\phi(t)})^{-1}\tilde\eta_{(i)}\right)\right),\tilde \eta_{(j)}\right),\, 0\le i<j\le n,
\end{equation}
where $(\cdot,\cdot)$ is the $L^{2}$-inner product on $A^{*,*}$ induced by the initial K\"ahler form $\omega_{0}$.
Then the local period map is defined via the local sections 
\begin{eqnarray}
\Omega_{(i)}(t)&=&\eta_{(i)}+\sum_{j=i+1}^{n}\Phi^{(i,j)}(t)\cdot \eta_{(j)} \label{intr defn of quasi Omega'} \\
&:\,=&\left[\mathbb H \left(e^{i_{\phi(t)}}\left((I+Ti_{\phi(t)})^{-1}\tilde\eta_{(i)}\right)\right)\right], \label{intr defn of quasi Omega}
\end{eqnarray}
of the Hodge bundles $\mathscr F^{n-i}$ on $B$, $0\le i\le n$. Here $[\cdot]=[\mathbb H(\cdot)]$ denotes the corresponding cohomology class.

We call the sections $\Omega_{(i)}(t)$ given by formula \eqref{intr defn of quasi Omega'} and \eqref{intr defn of quasi Omega} the sections of Hodge bundles defined from period matrices and deformation theory respectively.

We will prove that the period map $\Phi$ in \eqref{intr local pm} is an analytic map from the analytic space $B$, and that it still satisfies the Griffiths transversality condition
\begin{equation*}
\left(\Phi^{(i,j)}\right)_{\mu}^{\bullet}(t)
= \left(\Phi^{(i,i+1)}\right)_{\mu}^{\bullet}(t)\,\Phi^{(i+1, j)}(t),
\quad 0\le i<j\le n.
\end{equation*}
Here $(\cdot)_{\mu}^{\bullet}(t)=\frac{\partial(\cdot)}{\partial t_{\mu}}(t)$
for the local coordinates $(t_{\mu})_{1\le \mu\le N}$ on $B$.
It is worth noting that the definition of the local period map $\Phi$ in \eqref{intr local pm} does not make use of the Hodge theory on the nearby complex structures $X_{t}$, for $t\in B$ and $t\neq 0$. Instead, it relies only on the Hodge theory of $X = X_{0}$ and the Beltrami differential $\phi(t)$ on $X$. Consequently, various Hodge-theoretic properties of the nearby complex structures $X_{t}$, $t\in B$, can be established via the period map $\Phi$ in \eqref{intr local pm}.

Now we set the weight of the Hodge structure to be $n=2p$. In Sections~\ref{pp Section} and~\ref{app Hodge section}, we establish the existence of a real-analytic map, referred to as the Hodge map,
\begin{eqnarray}
H&:&H^{p,p}(X,\mathbb{R}) \times B \to H^{2p}(X,\mathbb{R}), \label{intr defn of H}\\
&&(\sigma,t) \mapsto \sum_{0\le i\le p-1} \alpha_{(i)}(\sigma,t)\cdot \eta_{(i)} + \sigma + \sum_{p+1\le j\le 2p} \overline{\alpha_{(j)}(\sigma,t)} \cdot \eta_{(j)}, \nonumber
\end{eqnarray}
such that $H(\sigma,t) \in H^{p,p}(X_t,\mathbb{R})$.
Here the functions $\alpha_{(i)}(\sigma,t)$ with values in row vectors are determined only by the blocks $\Phi^{(p,q)}(t)$, $0\le p\le q\le n$, as in \eqref{B=01} and \eqref{B=02}.

We are particularly interested in the case that $p=1$ and the weight $n=2$. Then the Hodge map is given by
\begin{eqnarray}
H&:&H^{1,1}(X,\mathbb{R}) \times B \to \bigcup_{t\in B}H^{1,1}(X_t,\mathbb{R}), \label{intr defn of H2}\\
&&(\sigma,t) \mapsto \alpha_{(0)}(\sigma,t)\cdot \eta_{(0)} + \sigma + \overline{\alpha_{(0)}(\sigma,t)} \cdot \eta_{(2)}, \nonumber
\end{eqnarray}
where $\alpha_{(0)}(\sigma,t)=\alpha_{(0)}(\alpha_{(1)},t)$ with $\sigma=\alpha_{(1)}\cdot \eta_{(1)}$ is the real analytic function which is determined by the equation
\begin{equation}\label{intr B=01 11-0}
\bar{\alpha_{(0)}}- \alpha_{(1)}\Phi^{(1,2)}(t)-\alpha_{(0)}\left(\Phi^{(0,2)}(t)-\Phi^{(0,1)}(t)\Phi^{(1,2)}(t)\right)=0.
\end{equation}

\noindent \textbf{K\"ahler cones and their $\nabla_{t_0}^{1,1}$-flat extensions}

Let $\mathcal X \to S$ be an analytic family of compact complex manifolds over an irreducible analytic base $S$ with a K\"ahler fiber $X_{t_{0}}$ for some $t_{0}\in S$.
The local differentiable trivializations of the family $\pi:\mathcal X\to S$ induce the Gauss--Manin connection
$$\nabla^{GM}:\mathcal H\longrightarrow \Omega_S^1\otimes_{\mathcal O_S}\mathcal H,$$
where $\mathcal H_t=H^2(X_t,\mathbb C)$ for every $t\in S$.

For each $t\in S$, let
$$P_t:H^2(X_t,\mathbb C)\longrightarrow H^{1,1}(X_t)$$
be the projection determined by the Hodge decomposition
$$H^2(X_t,\mathbb C)=H^{2,0}(X_t)\oplus H^{1,1}(X_t)\oplus H^{0,2}(X_t).$$
In particular,
$$P_{t_{0}}:H^2(X_{t_{0}},\mathbb C)\longrightarrow H^{1,1}(X_{t_{0}})$$
is the projection determined by the Hodge structure at the reference point $X_{t_{0}}$.

Inspired by the work of Demailly and Paun \cite{DP04}, we define the moving Hodge connection $\nabla^{1,1}$ by
$$\nabla^{1,1}|_{t}:=P_t\circ\nabla^{GM},$$
and the reference Hodge connection $\nabla_{t_0}^{1,1}$ by
$$\nabla_{t_0}^{1,1}|_{t}:=P_{t_{0}}\circ\nabla^{GM}.$$
Thus, $\nabla^{1,1}$ is determined by the varying Hodge decomposition on $X_t$, whereas $\nabla_{t_0}^{1,1}$ is determined by the fixed Hodge decomposition on $X_{t_{0}}$.

Note that if we fix $\sigma\in H^{1,1}(X_{t_{0}},\mathbb R)$ in \eqref{intr defn of H2}, then the section $H(\sigma,t)\in H^{1,1}(X_{t},\mathbb R)$ is $\nabla_{t_{0}}^{1,1}$-flat for $t$ near $t_{0}$.
Inspired by this, we first give the definition of the two flat extensions of K\"ahler cones.

\begin{definition}\label{extensions of the Kahler cone defn}
Let $\mathcal X \to S$ be an analytic family of compact complex manifolds over an irreducible analytic base $S$ with a K\"ahler fiber $X_{t_{0}}$ for some $t_{0}\in S$. We define the $\nabla_{t_{0}}^{1,1}$-flat and $\nabla^{1,1}$-flat extensions of the K\"ahler cone $\mathcal{K}_{t_{0}}$ on $X_{t_{0}}$ respectively by
\begin{eqnarray*}
\mathcal{K}^{\nabla_{t_{0}}^{1,1}}_{t_{0},t}&=&\left\{H(\sigma,t):\, \sigma\in \mathcal{K}_{t_{0}}\right\}
\subset H^{1,1}(X_{t},\mathbb{R}),\\
\mathcal{K}^{\nabla^{1,1}}_{t_{0},t}&=&\left\{H(\sigma_{t},t):\, \nabla^{1,1}H(\sigma_{t},t)=0,\,\sigma_{t_{0}}\in \mathcal{K}_{t_{0}}\right\}
\subset H^{1,1}(X_{t},\mathbb{R})
\end{eqnarray*}
where $t\in U_{t_{0}}$ for some neighborhood $U_{t_{0}}$ of $t_{0}$.
\end{definition}

A key property of the $\nabla_{t_0}^{1,1}$-flat extensions of the K\"ahler cone is that they admit positive definite representatives in a uniform neighborhood.

\begin{proposition}\label{intr pos def of Kahler}
The classes in $\mathcal{K}^{\nabla_{t_0}^{1,1}}_{t_{0},t}$ admit positive definite representatives provided that there exists a Beltrami differential $\phi(t)\in A^{0,1}(X_{t_{0}},\Tan X_{t_{0}})$ with supremum operator norm
$$
\|\phi(t)\|^{E}<1 \quad \text{(see Equation~\eqref{son E})}
$$
such that
$$
X_{t}=(X_{t_{0}})_{\phi(t)},
$$
that is, there exists a diffeomorphism $d_{t}:\, X_{t}\to X_{t_{0}}$ which induces $\phi(t)$.
\end{proposition}

See the proof of Theorem \ref{Kahler cone invariance} for details.

In Theorem~0.9 of \cite{DP04}, Demailly and Paun proved that if $X$ is a compact K\"ahler manifold, then its K\"ahler cone $\mathcal K$ is one of the connected components of the set $\mathcal P$ consisting of real $(1,1)$-cohomology classes $\alpha$ that are numerically positive on every irreducible analytic subset $Z\subset X$, namely,
$$\alpha^{\dim Z}\cdot Z=\int_Z\alpha^{\dim Z}>0.$$

They further considered a smooth family $\pi:\mathcal X\to S$ of compact K\"ahler manifolds. By Barlet's theory of cycle spaces \cite{Barlet75}, there exists a countable union
$$S'=\bigcup_{\nu}S_{\nu}$$
of proper analytic subsets $S_{\nu}\subset S$ such that, for every $t_0\in S\setminus S'$, every analytic cycle $Z_{t_0}\subset X_{t_0}$ deforms to a family of analytic cycles $Z_t\subset X_t$. If $\alpha(t)\in H^{1,1}(X_t,\mathbb R)$ is a $\nabla^{1,1}$-flat section, namely,
$$\nabla^{1,1}\alpha(t)=0,$$
then, for every such family $Z_t\subset X_t$, one has
$$\alpha(t)^{\dim Z_t}\cdot Z_t=\mathrm{constant}.$$
Consequently, the K\"ahler cones
$$\mathcal K^{\nabla^{1,1}}_{t_0,t}=\mathcal K_{t}$$
are identified by the parallel transport induced by the moving Hodge connection $\nabla^{1,1}$ for all $t_{0},t\in S\setminus S'$.

In Section \ref{Kahler section}, we give the precise formula of the $\nabla^{1,1}$-flat extensions $\mathcal K^{\nabla^{1,1}}_{t_0,t}$ of the K\"ahler cone.

\begin{proposition}\label{intr DP Kahler extensions}
Let $t_0\in S$. There exists a family of linear maps
$$M(t):H^{1,1}(X_{t_0},\mathbb R)\longrightarrow H^{1,1}(X_{t_0},\mathbb R),$$
uniquely determined by the blocks $\Phi^{(0,1)}(t)$, $\Phi^{(0,2)}(t)$, $\Phi^{(1,2)}(t)$ of the period map $\Phi(t)\in N_{-}$, such that the $\nabla^{1,1}$-flat extensions $H(\sigma_t,t)$ of $\sigma_{t_{0}}\in H^{1,1}(X_{t_0},\mathbb R)$ is given by
\begin{equation}\label{intr formula of Kahler extensions}
\sigma_t=\exp\left(\int_{\gamma_{t_0,t}}M(s)\,ds\right)\sigma_{t_{0}},
\end{equation}
where the integral $\int_{\gamma_{t_0,t}}$ depends on the path $\gamma_{t_0,t}$ in $S$ joining $t_0$ to $t$.
Consequently, the $\nabla^{1,1}$-flat extension
$$\mathcal K^{\nabla^{1,1}}_{t_0,t}=\left\{H(\sigma_t,t):\,\sigma_{t_0}\in\mathcal K_{{t_0}}
\right\}$$
is explicitly determined by the Beltrami differential
$$\phi(t)\in A^{0,1}(X_{t_0},T^{1,0}X_{t_0})$$
along the path $\gamma_{t_0,t}$.
\end{proposition}

By combining Theorem~0.9 of Demailly--Paun in \cite{DP04} with our main results in Section \ref{Kahler section}, we obtain the following complete characterization of the deformations of K\"ahler cones, which constitutes one of the main results of this paper.

\begin{theorem}\label{intr strong Kahler cone invariance}
Let $\mathcal X \to S$ be a deformation of compact K\"ahler manifolds over an irreducible analytic base $S$. Let $\mathcal{K}_{t}\subset H^{1,1}(X_{t},\mathbb{R})$ be the K\"ahler cones for $t\in S$.
Then the following statements hold:

\emph{(1)} (Upper semicontinuity of K\"ahler cones) For any $t_{0}\in S$, the $\nabla_{t_{0}}^{1,1}$-flat extensions $\mathcal{K}^{\nabla_{t_0}^{1,1}}_{t_{0},t}$
of the K\"ahler cone $\mathcal{K}_{t_{0}}$ exist and satisfy
$$
\mathcal{K}^{\nabla_{t_0}^{1,1}}_{t_{0},t}=\left\{H(\sigma,t):\, \sigma\in \mathcal{K}_{t_{0}}\right\}\subset \mathcal{K}_{t},\quad t\in U_{t_{0}},
$$
for some sufficiently small neighborhood $U_{t_{0}}$ of $t_{0}$. Moreover, the positive definite representative of any $H(\sigma,t)\in \mathcal{K}^{\nabla_{t_0}^{1,1}}_{t_{0},t}$ is given by
$$
\alpha_{(0)}(\sigma,t)\cdot\tilde \eta^{\omega_{\sigma}}_{(0)}+\omega_{\sigma}+\bar{\alpha_{(0)}(\sigma,t)}\cdot \tilde \eta^{\omega_{\sigma}}_{(2)},
$$
where $\omega_{\sigma}$ is the positive definite representative of $\sigma$, and $\tilde \eta^{\omega_{\sigma}}_{(i)}$, $0\le i\le 2$, are the harmonic representatives of $\eta_{(i)}$ with respect to the Laplacians induced by $\omega_{\sigma}$.

\emph{(2)} There exists a countable union $S' = \bigcup_{\nu} S_{\nu}$ of analytic subsets $S_{\nu} \subset S$ such that
$$
\mathcal K^{\nabla^{1,1}}_{t_0,t}=\left\{H(\sigma_t,t):\,\sigma_{t_0}\in\mathcal K_{{t_0}},\, \sigma_{t} \text{ is given by }\eqref{intr formula of Kahler extensions}
\right\}= \mathcal{K}_{t}
$$
for $t_{0}\in S\setminus S'$ and $t\in U_{t_{0}}\setminus S'$.
\end{theorem}

The two connections $\nabla_{t_0}^{1,1}$ and $\nabla^{1,1}$ provide two different descriptions of the variation of the K\"ahler cones. The $\nabla_{t_0}^{1,1}$-flat extension is defined uniformly over the whole family and satisfies
$$\mathcal K^{\nabla_{t_0}^{1,1}}_{t_0,t}\subset\mathcal K_{t}$$
for every nearby $t\in S$, independently of whether new analytic cycles appear on $X_t$. Thus, it gives a uniform inclusion property for the K\"ahler cones throughout the family.

On the other hand, the moving Hodge connection $\nabla^{1,1}$ preserves the numerical intersections with all analytic cycles that deform along the family. Consequently, after removing the countable union $S'\subset S$ of proper analytic subsets arising from the cycle spaces, its parallel transport identifies the K\"ahler cones:
$$\mathcal K^{\nabla^{1,1}}_{t_0,t}=\mathcal K_{t},\qquad t_0,t\in S\setminus S'.$$
Hence, the fixed Hodge connection $\nabla_{t_0}^{1,1}$ gives a uniform comparison valid even in the presence of jumping cycles, whereas the moving Hodge connection $\nabla^{1,1}$ gives an isomorphism of the K\"ahler cones on the locus where the relevant analytic cycles deform.

\noindent\textbf{Large-scale K\"ahler stability}

Since our method of extending $(1,1)$-classes uses the period matrices \eqref{period matrices defn}, which are determined by the Beltrami differential, and the positivity of the extended $(1,1)$-classes is likewise determined by the Beltrami differential; cf.\ Proposition~\ref{intr pos def of Kahler}, we obtain the  K\"ahler stability on large scales in the following sense. 

For an analytic family $f:\,\X \to S$ of compact complex manifolds, we define the open subset $S_{t_{0},c}\subset S$ to the connected component containing $t_{0}$ of 
$$\left\{t\in S\bigg|\begin{array}{ll}\, \text{$X_{t}=(X_{t_{0}})_{\phi(t)}$ for some $\phi(t)\in A^{0,1}\left( X_{t_{0}},\mathrm T^{1,0}X_{t_{0}}\right)$}\\ \text{with the supremum operator norm $\|\phi(t)\|<c$}.\end{array}\right\}.
$$
Here the supremum operator norm $\|\cdot\|= \|\cdot\|^{E}$ is defined as in Equation~\eqref{son E}, which is independent of the Hermitian metrics on $X$.

\begin{theorem}\label{intr global Kahler}
Let $f:\,\X \to S$ be an analytic family of compact complex manifolds over a connected complex manifold $S$, with a fiber $X_{t_0}$ being K\"ahler.
Then there exists a constant $c_{0}$, which depends only on the Hodge numbers $h^{2,0}(X_{t_0})=h^{0,2}(X_{t_0}), h^{1,1}(X_{t_0})$, such that all the fibers $X_{t}$ are K\"ahler for $t\in S_{t_{0},c_{0}}$. Moreover, the $\nabla_{t_0}^{1,1}$-flat extensions $\mathcal{K}^{\nabla_{t_0}^{1,1}}_{t_{0},t}\subset \mathcal{K}_{t}$ in Theorem~\ref{intr strong Kahler cone invariance} exist for all $t\in S_{t_{0},c_{0}}$.
\end{theorem}

The large-scale K\"ahler stability established in Theorem~\ref{intr global Kahler} plays a crucial role in our approach: it provides the key ingredient for a new proof of Siu's theorem \cite{Siu83} on the K\"ahlerness of all K3 surfaces, and it further allows us to treat more general situations in our forthcoming paper \cite{LS26-3}.

\noindent\textbf{Generalized Green criteria for algebraic approximation}

Next, we apply the Hodge map in \eqref{intr defn of H2} to address the following Kodaira problem. Our approach differs from the classical framework in two essential ways: it is formulated on the (possibly singular) Kuranishi base and thus applies to obstructed deformations, and it incorporates all higher-order terms of the Beltrami differential, rather than only its first-order part.

\begin{theorem}\label{intr alg app main theorem}
Let $(X,\omega_{0})$ be a compact K\"ahler manifold. If there exists $\zeta_{0}=[\tilde\zeta_{0}]\in H^{1,1}(X,\mathbb R)$ with harmonic representative $\tilde\zeta$ such that the map
\begin{equation}\label{intr iphi}
[\mathbb H \left(i_{\phi(\cdot)}\tilde\zeta_{0}\right)]:\, B\to H^{0,2}(X)
\end{equation}
is open when the Kuranishi base $B$ is sufficiently small, then $X$ can be strongly approximated by projective manifolds.
\end{theorem}

Here we say that $X$ can be strongly approximated by projective manifolds provided that the locus
$$
\{t\in B:\, X_{t} \text{ is projective}\}
$$
is dense in a neighborhood of $0$. In fact, condition \eqref{intr iphi} implies that the image $H(H^{1,1}(X,\mathbb R)\times B)$ of the Hodge map \eqref{intr defn of H2} is open in $H^{2}(X,\mathbb R)$, so that the locus
$$
\{t\in B:\, H(\sigma,t)\in H^{2}(X,\mathbb Q)\}
$$
is dense near $0$ in $B$.

When the deformation of $X$ is unobstructed, Green's density criterion, see Proposition~5.20 in~\cite{Voisin2}, is recovered as the first-order term of~\eqref{intr iphi},
$$
i_{(\cdot)}\zeta_{0}:\, H^{1}(X,\Theta_X)\to H^{0,2}(X).
$$
Green's criterion was also proved independently by Buchdahl and has found notable applications to the Kodaira problem, that is, the algebraic approximation of compact K\"ahler manifolds; see, for example, \cite{Buchdahl06}, \cite{Buchdahl08}, \cite{Cao15}, and \cite{Graf17}.

As noted by Voisin in Chapter~5.3.4 of~\cite{Voisin2}, any direct generalization of the Green density criterion based solely on the first-order term of~\eqref{intr iphi} fails for $(p,p)$-classes with $p\ge 2$. By contrast, our formulation takes into account all higher-order terms in the Beltrami differential $\phi(t)$ through the Hodge map \eqref{intr defn of H}, which allows us to obtain effective criteria for general $(p,p)$-classes on the full Kuranishi base, including the obstructed case.

\begin{theorem}\label{intr app by Hodge main theorem}
Let $X$ be a compact K\"ahler manifold. 
If there exists a real cohomology class $\sigma_0=[\tilde \sigma_0] \in H^{p,p}(X,\mathbb R)$ with harmonic representative $\tilde \sigma_0$ such that the map
\begin{equation}\label{intr iphis}
\left(\left[\mathbb H \left(i_{\phi(\cdot)}\tilde \sigma_0\right)\right],\cdots,\left[\mathbb H \left(i^{p}_{\phi(\cdot)}\tilde \sigma_0\right)\right]\right):\,  B\to H^{p-1,p+1}(X)\oplus \cdots \oplus H^{0,2p}(X)
\end{equation}
is an open map when the Kuranishi base $B$ is sufficiently small, then $\sigma_0$ can be strongly approximated by nearby Hodge classes (see Definition \ref{defn of app by Hodge}).
\end{theorem}

\noindent\textbf{Hodge locus and variational Hodge conjecture}

Finally we consider the problem of Hodge locus and variational Hodge conjecture.

\begin{theorem}\label{intr Hodge locus main}
Let $X$ be a compact K\"ahler manifold and $\sigma=[\tilde\sigma] \in H^{2p}(X,\mathbb Q)\cap H^{p,p}(X)$ be a Hodge class with harmonic representative $\tilde\sigma$. 
Then the Hodge locus
$$B_{\sigma}^{p}=\left\{t\in B:\, \sigma\in H^{p,p}(X_{t})\right\},$$
defined by the nearby Hodge structures, admits the intrinsic characterization on $X$
\begin{equation}\label{intr eqn of Hodge locus}
B_{\sigma}^{p}=\left\{t\in B:\, \mathbb H \left({i_{\phi(t)}}\left((I+Ti_{\phi(t)})^{-1}\tilde\sigma\right)\right)=0\right\},
\end{equation}
which is an analytic subset of $\Delta \subset \mathbb{H}^{0,1}(X, \Tan X)$.
\end{theorem}

With the above characterization of the Hodge locus, we obtain equivalent conditions for the variational Hodge conjecture. 
Here we restrict to the case where $Z \subset X$ is a smooth analytic subvariety of codimension $p$, and the conjecture asserts that the Hodge locus $B^{p}_{\sigma_{Z}}$ coincides with the locus $\mathrm{Def}(X,Z)$ of deformations of $Z$ as submanifolds within deformations of $X$.

\begin{theorem}\label{intr vhc main}
Let $X$ be a compact K\"ahler manifold, and let $Z \subset X$ be a smooth analytic subvariety. Then the variational Hodge conjecture holds for $X$ at $Z$ if and only if the implication
\begin{equation}\label{intr vhc criterion}
\mathbb{H}_{N_{Z|X}}\left( \phi(t)\big|_{N_{Z|X}} \right) \neq 0 \;\Longrightarrow\; \mathbb{H}\left(i_{\phi(t)} \tilde{\sigma}_{Z} \right) \neq 0
\end{equation}
holds for every $t \in B$, where $\sigma_{Z} = [\tilde{\sigma}_{Z}]$ denotes the Hodge class associated to $Z$.
\end{theorem}

Since the celebrated work of Bloch~\cite{Bloch72}, semi-regularity has provided a classical sufficient condition for the variational Hodge conjecture. Further examples satisfying this condition have been studied in~\cite{DK16} and~\cite{Kloosterman22}. By contrast, Theorem~\ref{intr vhc main} gives an explicit necessary and sufficient criterion in terms of the Beltrami differential, without assuming semi-regularity and allowing the Kuranishi base to be singular. In particular, it applies to genuinely nontrivial situations in which $\mathrm{Def}(X,Z)=B_{\sigma_Z}^{p}\subsetneq B$ and provides a concrete method for constructing new examples of the variational Hodge conjecture.

\noindent\textbf{Final remark}

The results of this paper show that the variation of Hodge structures, K\"ahler cones, Hodge classes, and analytic cycles can be described explicitly in terms of the Hodge theory and elliptic operator theory of a fixed reference fiber together with the associated Beltrami differential.

Traditionally, Beltrami differentials have been used mainly to study integrability and unobstructedness of deformations. Our framework shows that they can also be used to investigate variations of Hodge structures, K\"ahler cones, and Hodge classes even when the deformation space is a singular Kuranishi base. 
Since Beltrami differentials naturally belong to the infinite-dimensional space
$$
A^{0,1}(X,T^{1,0}X),
$$
a natural further direction is to seek smaller and more effective parameter spaces for the present framework, possibly finite-dimensional in favorable situations, in the spirit of Kuranishi reduction and homotopy transfer in Kodaira--Spencer deformation theory, while retaining the explicit geometric information encoded by the Beltrami differential.

\noindent\textbf{Organization of the paper}

This paper is organized as follows. Section~\ref{SHB} recalls the global construction of Hodge bundles from \cite{LS26-1}, obtained both from the matrix representation of the image of the period map and from deformation theory. 
Section~\ref{pp Section} provides a real-analytic parametrization of $(p,p)$-classes via a Hodge map that incorporates explicit higher-order Beltrami terms. 
In Section~\ref{Kahler section}, we study the upper semicontinuity of K\"ahler cones under deformation, describing their behavior under parallel transport via the Gauss--Manin connection and refining results of \cite{DP04}. 
As a preparatory step, we also prove the K\"ahler stability theorem without assuming unobstructedness, which serves mainly as a warm-up for the cone-theoretic analysis. 
In Section~\ref{large KS}, we establish a global K\"ahler stability theorem together with the existence of $\nabla_{t_0}^{1,1}$-flat extensions of K\"ahler cones over a large region in the base.
Section~\ref{algebraic app section} formulates a generalization of Green’s criterion for strong algebraic approximation in terms of the Beltrami differential. Section~\ref{app Hodge section} gives a criterion for approximating real $(p,p)$-classes by nearby Hodge classes. Finally, in Section~\ref{Hodge locus Hodge conj}, we give an intrinsic analytic description of the Hodge locus and formulate a criterion for the variational Hodge conjecture using the Beltrami differential along the normal bundle.



\section{Sections of Hodge bundles}\label{SHB}

In this section, we recall the global construction of Hodge bundles in \cite{LS26-1}, which are constructed both from the matrix representation of the image of the period map and from deformation theory.

In this paper, the period domain $D$ is the set of Hodge structures
\begin{eqnarray}
&&H=\bigoplus_{p+q=n}H^{p,q},\, H^{p,q}=\bar{H^{q,p}}\\
&\Longleftrightarrow&F^{n}\subset \cdots \subset F^{0}=H,\, F^{p}\oplus \bar{F^{n-p+1}}=H
\end{eqnarray}
of certain type on the complex linear space $H$, which are not necessarily polarized.
Moreover, the period domain $D$ can be realized as the flag domain $$D = G_\mathbb{R} / V$$ in the flag manifold $$\check{D}= G_{\C} / B,$$ where $G_\C$ is a complex Lie group with $B\subset G_\C$ a parabolic subgroup and $G_\mathbb{R}$ is a real Lie subgroup of $G_\C$ with $V = B \cap G_{\mathbb{R}}$ a compact Lie subgroup of $G_\mathbb{R}$.

The Hodge structure $(H=\oplus_{p+q=n}H_o^{p,q})$ at a fixed point $o$ in $D \subseteq \check{D}$ induces a Hodge structure of weight zero on the Lie algebra $\mathfrak{g}$ of $G_{\mathbb{C}}$ as
$$\mathfrak{g} = \bigoplus_{k \in \mathbb{Z}} \mathfrak{g}^{k, -k}, \quad \mathfrak{g}^{k, -k} = \{X \in \mathfrak{g} \mid X H_o^{p, q} \subseteq H_o^{p+k, q-k}, \ \forall\, p+q=n\}.$$ 
Then the Lie algebra $\mathfrak{b}$ of $B$ can be identified with $\bigoplus_{k \geq 0} \mathfrak{g}^{k, -k}$ and the holomorphic tangent space $\mathrm{T}^{1,0}_o{D}$ of ${D}$ at the base point $o$ is naturally isomorphic to
$$\mathfrak{g} / \mathfrak{b} \simeq \oplus_{k \geq 1} \mathfrak{g}^{-k,k} \triangleq \mathfrak{n}_-.$$

We denote the corresponding unipotent group by 
$$N_- = \exp(\mathfrak{n}_-)$$ 
which is identified to $N_-(o)$, the unipotent orbit of the base point $o$, and is considered as a complex Euclidean space inside $\check{D}$.
It can be proved that the complex Euclidean space $N_-$ is Zariski open in $\check{D}$, using the method in the proof of Proposition 1.8 in \cite{LS26-1}.

Now we consider a family of compact K\"ahler manifolds $f:\, \mathcal X \to \Delta$ over a polydisk in $\C^{N}$. Throughout this paper, the radius $\epsilon$ of the polydisk $\Delta$ is taken to be the same in each coordinate direction, and it may be chosen arbitrarily small. Then we have a period map
$$\Phi:\, \Delta \to N_{-}\cap D,$$
with values in $N_{-}$ for small $\Delta$.

To give an explicit description of the image of the period map $\Phi$ in $N_{-}$, we fix an adapted basis of the Hodge decomposition at the base point $0\in \Delta$ and denote it by
$$
\eta = \{\eta_{(0)}^T, \eta_{(1)}^T, \cdots, \eta_{(n)}^T\}^T,
$$
where for each $0 \leq p \leq n$, $\eta_{(p)}$, viewed as a column vector with entries in $H$, is a basis of $ H^{n-p,p}(X_{0}) $.
Here the transpose operator $T$ is used solely to arrange the basis elements so that $\eta$ is regarded as a column vector.

Under the above basis $\eta$, the image of the local period map $\Phi:\, \Delta \to N_{-}\cap D$,
$$
\P(t) = \left[\begin{array}{ccc}
I & & \left( \Phi^{(p,q)}(t) \right)_{p < q} \\
& \ddots & \\
& & I
\end{array}\right] \in N_-,\,t\in \Delta
$$
gives a basis
\begin{equation}\label{lm section}
  \Omega_{(p)}(t)=\eta_{(p)}+\sum_{k\ge 1} \Phi^{(p,p+k)}(t)\cdot \eta_{(p+k)},
\end{equation}
of the filtration $ F^{n-p} H^{n}(X_{t}, \C) $ modulo $ F^{n-p+1} H^{n}(X_{t}, \C) $ for $ 0 \leq p \leq n $. 

We refer to $\Omega_{(p)}(t)$ in \eqref{lm section} as holomorphic sections of the Hodge bundles $\mathcal{F}^{n-p}$ over $\Delta$.

Note that the Hodge filtration $F^{\bullet}$ on the cohomology group $H^{n}(X_{t}, \C)$ is induced by the complex structure of $X_{t}$, which is determined by the Beltrami differential $\phi(t)\in A^{0,1}(X,\Tan X)$. The family $\phi(t)$ of the Beltrami differentials depends holomorphically on $t\in \Delta$ with $\phi(0)=0$, and satisfies the integrability condition
$$\bar{\partial} \phi(t) - \frac{1}{2}[\phi(t), \phi(t)] = 0.$$  
In fact, for any $p\in X$, we can
pick a local holomorphic coordinate $(U,z_1,...,z_n)$ near $p$ with $z_{i}(p)=0$, and then the complex structure of $X_{t}$ is equivalent to the following
\begin{align}
&{\mathrm{T}^{*}}^{1,0}(X_{t})|_{U}=\text{Span}_{\mathbb{C}}\{dz_1+\phi(t)
dz_1,\cdots,dz_n+\phi(t) dz_n\}, \label{basis of cotang}\\
&{\mathrm{T}^{*}}^{0,1}(X_{t})|_{U}=\text{Span}_{\mathbb{C}}\{d\bar{z}_1+\bar\phi(t)
d\bar{z}_1,\cdots,d\bar{z}_n+\bar\phi(t) d\bar{z}_n\} \label{basis of cotang bar}
\end{align}

For $\phi\in A^{0,1}(X,\Tan X)$, let 
$$i_\phi:\, A^{p,q}(X)\to A^{p-1,q+1}(X)$$
be the contraction map, and $\rho_{\phi}=e^{i_\phi}$ be the exponential of the contraction map,
\begin{align} \label{exponentialphi}
\rho_{\phi}(\sigma)= e^{i_\phi}\sigma= \sum_{k\geq
0}\frac{1}{k!}i_{\phi}^k\sigma,
\end{align}
for $\sigma \in A^{p,q}(X)$.

\begin{definition}\label{son defn}
Let $\{(U;z)\}$ be a finite cover of a compact complex manifold $X$ with holomorphic coordinates, and for each open subset $U$ there exists an open subset $V$ with $V\subset \bar{V}\subset U$ such that $\{V\}$ still covers $X$. Let $\phi \in A^{0,1}(X,\mathrm T^{1,0}X)$ be a Beltrami differential with local expression
$$\phi|_{U}=\sum_{ij}\phi_{\bar j}^{i}(z){\partial_{i}}\otimes d\bar{z_{j}},$$
Then the supremum operator norm of $\phi$ is defined by
\begin{equation}\label{son E}
\|\phi\|^{E}=\max_{U}\max_{z\in \bar{V}} \sigma_{\mathrm{max}}\left(\phi_{\bar j}^{i}(z)\right),
\end{equation}
where $\sigma_{\mathrm{max}}(A)$ denotes the largest singular value of a complex square matrix $A$, i.e. the largest square root of the eigenvalues of $\bar{A}^{T}A$. 

\end{definition}
 
In \cite{LS26-1} we also define the supremum operator norm of the Beltrami differential 
$\phi$ on a compact K\"ahler manifold $X$ by
\begin{equation}\label{supremum operator norm}
\|\phi\|^{\omega} := \sup_{x\in X} \|\phi\|^{\omega}_{2,x},
\end{equation}
where the pointwise operator norm is given by
\begin{equation*}
\|\phi\|^{\omega}_{2,x} := \sup_{v\in \mathrm T_{x}^{0,1}X}
\frac{\|\phi(v)\|_{x}^{\omega}}{\|v\|_{x}^{\omega}}, \quad x\in X,
\end{equation*}
and $\|\cdot\|_{x}^{\omega}$ denotes the Hermitian metric on 
$\mathrm T_{x}^{1,0}X$ and $\mathrm T_{x}^{0,1}X$ induced by the K\"ahler form $\omega$ on $X$.

In fact, the norm $\|\phi\|^{E}$ is equal to the norm in \eqref{supremum operator norm} with the K\"ahler metric $\omega$ replaced by the standard Euclidean metric on $\mathrm TX|_{U}$. The superscript $E$ therefore refers to the Euclidean metric.

Since $X$ is compact, any two smooth Hermitian metrics on $TX$ are uniformly equivalent. Hence there exist constants $0<c\le C$, depending only on the chosen covering $\{V\subset U\}$ and the K\"ahler form $\omega$, such that
$$
c\,\|\phi\|^{\omega}\le \|\phi\|^{E}\le C\,\|\phi\|^{\omega}.
$$
In \cite{LS26-1}, we restricted to polarized manifolds, where the K\"ahler form $\omega$ is fixed, and thus wrote $\|\phi\|=\|\phi\|^{\omega}$. In the present paper, we allow arbitrary K\"ahler forms on $X$ and therefore adopt the following convention.

\noindent\textbf{Convention.} \emph{Throughout this paper, we write $\|\phi\|=\|\phi\|^{E}$, which is independent of the K\"ahler forms on $X$.}

In \cite{LS26-1}, we have proved the following properties of the operator $\rho_{\phi}=e^{i_\phi}$ for the integrable Beltrami differential $\phi\in A^{0,1}(X,\Tan X)$ (see also \cite{LR}, \cite{LRY} and \cite{LZ18}):
\begin{itemize}
\item The restriction to the Hodge filtration
$$e^{i_\phi}:\, F^pA^{p+q}(X)\rightarrow F^pA^{p+q}(X_\phi)$$
is an isomorphism for $p,q\ge 0$;

\item For any smooth form $\sigma \in A^{p,q}(X)$, the
corresponding form
$\rho_{\phi}(\sigma)=e^{i_{\phi}}(\sigma)$ in $F^pA^{p+q}(X_\phi)$ is
$d$-closed, if and only if
\begin{equation}\label{obstructionequation2}
\left\{ \begin{array}{lr} {\partial}\sigma=0,\\
\bar{\partial}\sigma =-\partial(\phi\lrcorner \sigma).
\end{array} \right.
\end{equation}

\item Let $\phi\in
A^{0,1}(X,\mathrm T^{1,0}X)$ be a Beltrami differential with
the supremum operator norm $\|\phi\|<1$, and $T :\,= \bar\partial^{*}G\partial$ be the operator from the harmonic theory on compact complex manifolds. Then the operator
$I+Ti_{\phi}$ is invertible on the Hilbert space of $L^2$ integrable $(p,q)$-forms. In particular, for any harmonic form $\sigma_0\in \mathbb H^{p,q}(X)$, the resulting $(p,q)$-form $(I+Ti_{\phi})^{-1}\sigma_0$ is the solution of the equations
\eqref{obstructionequation2}, and hence
\begin{align*}
\sigma(\phi)=e^{i_{\phi}}\left((I+Ti_{\phi})^{-1}\sigma_0\right) \in F^pA^{p+q}(X_\phi)
\end{align*}
is $d$-closed with $\sigma(0)=\sigma_0$.
\end{itemize}
 
For the adapted
basis $\eta = \{\eta_{(0)}^T, \eta_{(1)}^T, \cdots, \eta_{(n)}^T\}^T$ of the Hodge decomposition at the base point $0\in \Delta$, we choose $\tilde \eta_{(p)}$ as a basis of harmonic $(n-p,p)$-forms on $X=X_{0}$ representing the cohomology class $\eta_{(p)} = [\tilde \eta_{(p)}]$ at the base point for $0\le p\le n$. Then we have the following sections $\tO_{(p)}$ of the Hodge bundles $\mathcal{F}^{n-p}$ over $\Delta$,
\begin{align}
\tO_{(p)}(t)=&\left[\mathbb H \left(e^{i_{\phi(t)}}\left((I+Ti_{\phi(t)})^{-1}\tilde \eta_{(p)}\right)\right)\right]\nonumber \\
=&\eta_{(p)} + \sum_{k\ge 1}\frac{1}{k!}\left[\mathbb H\left(i_{\phi(t)}^k (I+Ti_{\phi(t)})^{-1}\tilde \eta_{(p)} \right)\right], \label{def section} 
\end{align}
for $0\le p\le n$.
Here $$\mathbb H:\, A^{p,q}(X)\to \mathbb H^{p,q}(X)$$ denotes the harmonic projection on $X$.

By comparing the constant terms of the expansions of $\Omega_{(p)}(t)$ in~\eqref{lm section} and $\tO_{(p)}(t)$ in~\eqref{def section}, we have the following theorem.

\begin{theorem}\label{main bundle}
Let the notations be as above. Then the sections $\Omega_{(p)}(t)$ in~\eqref{lm section}, defined via the matrix representation of the image $\P(t)$ in $N_-$, coincide with the sections $\tO_{(p)}(t)$ in~\eqref{def section}, which are defined using the Beltrami differential $\phi(t)$ associated with the complex structure $X_{t} = X_{\phi(t)}$. In particular, we have that
\begin{equation}\label{contraction k=Phi k}
\frac{1}{k!}\mathbb H\left(i_{\phi(t)}^k (I+Ti_{\phi(t)})^{-1}\tilde \eta_{(p)} \right)=\Phi^{(p,p+k)}(t)\cdot \tilde\eta_{(p+k)},
\end{equation}
and the sections $\tO_{(p)}(t)$ are holomorphic for all $t\in \Delta$.
\end{theorem}
\begin{remark}\label{intrinsic of pm}
Equation~\eqref{contraction k=Phi k} may be viewed as providing an intrinsic definition of the period map
$$\Phi:\, \Delta \to N_{-}\cap D, \, t\mapsto \left( \Phi^{(p,q)}(t) \right)_{0\le p,q\le n},$$
which depends only on the K\"ahler properties of the base manifold $X$ and the integrable Beltrami differential $\phi(t)\in A^{0,1}(X,\Tan X)$ on $X$. The strength of equation~\eqref{contraction k=Phi k} lies in the fact that it does not require the K\"ahlerness of $X_{t}$, which is used in the local study of the period map by Griffiths in \cite{Griffiths69}.
\end{remark}

In the following, we will frequently use the derivatives of the matrix blocks appearing in the image of the period map in $N_{-}$. For convenience, we adopt the notation
$$
(\Psi)_{\mu}^{\bullet}(t) = \frac{\partial \Psi}{\partial t_{\mu}}(t) = \left( \frac{\partial \psi_{ij}}{\partial t_{\mu}}(t) \right)_{\substack{1 \le i \le m \\ 1 \le j \le l}}, \quad 1 \le \mu \le N,
$$
where $\Psi(t) = \left( \psi_{ij}(t) \right)_{\substack{1 \le i \le m \\ 1 \le j \le l}}$ is a matrix-valued function.

\begin{lemma}\label{lm derivative lemma}
For the period map
$$\Phi:\, \Delta \to N_{-}\cap D, \, t\mapsto \left( \Phi^{(p,q)}(t) \right)_{0\le p,q\le n},$$
we have that 
\begin{align}\label{lm derivative}
(\Phi^{(p, p+i)})_{\mu}^{\bullet}(t) = (\Phi^{(p, p+1)})_{\mu}^{\bullet}(t) \cdot \Phi^{(p+1, p+i)}(t),
\end{align}
for any $0\le  p <p+i\le n$ and $1\le \mu\le N$.
\end{lemma}

The proof of Lemma~\ref{lm derivative lemma} consists of applying Griffiths transversality \cite{Griffiths69} to the matrix representation of the image of the period map in $N_-\cap D$. See Section~2 of \cite{LS26-1} for details.

From Lemma \ref{lm derivative lemma}, we have the following lemma on the estimates of the blocks of the image of the period map.

\begin{lemma}\label{order0}
Let the notations be as above. Then we have the following expansions of the blocks,
\begin{align}\label{lm expansion}
 \P^{(p, p+i)}(t)=O(|t|^{i}), \text{ for } 0\le p<p+i\le n.
\end{align}
\end{lemma}
\begin{proof}
At $t=0$, the matrix representation in $N_-$ of $\P(0)$ is the identity matrix, and hence $\P^{(p, p+i)}(0)=0$ for $0\le p< p+i\le n$. This implies that the expansions of the blocks $\P^{(p, p+i)}(t)$ have no constant terms, and that equation \eqref{lm expansion} holds for $i=1$, i.e.
\begin{equation}\label{lm expansion 1}
\P^{(p, p+1)}(t)=O(|t|).
\end{equation}

Now we prove the proposition inductively on $i$. We may assume that equation \eqref{lm expansion} holds for $i=s$, i.e. $\P^{(p, p+s)}(t)=O(|t|^{s})$, where $1\le s < n$. Let us consider the case for $i=s+1$.

By equation \eqref{lm derivative} in Lemma \ref{lm derivative lemma}, we have
\begin{align*}
\left(\Phi^{(p,p+s+1)}\right)_\mu^\bullet(t)= \left(\Phi^{(p,p+1)}\right)_\mu^\bullet(t) \Phi^{(p+1, p+s+1)}(t),
\end{align*}
where $\P^{(p+1, p+s+1)}(t)=O(|t|^{s})$ by induction, and, by \eqref{lm expansion 1}, $\left(\Phi^{(p,p+1)}\right)_\mu^\bullet(t)$ is bounded when $t$ is near the origin. Therefore we conclude that
$$\left(\Phi^{(p,p+s+1)}\right)_\mu^\bullet(t)=O(|t|^{s}),$$
which implies that
$$\P^{(p, p+s+1)}(t)=O(|t|^{s+1}),$$
since $\P^{(p, p+s+1)}(t)$ has no constant term. This finishes the proof by induction.
\end{proof}

From the proof of the above lemma, we obtain the following stronger expansion of the blocks:
\begin{align}\label{lm expansion strong}
 \P^{(p, p+i)}(t)=O\left(|t \P^{(p+1, p+i)}(t)|\right), \quad \text{for } 0 \le p < p+i \le n.
\end{align}


\section{Characterization of \texorpdfstring{$(p,p)$}{(p,p)}-classes}\label{pp Section}
In this section, we establish a real-analytic parametrization for $(p,p)$-classes in a family of compact K\"ahler manifolds. Specifically, given any $(p,p)$-class on the central fiber, we construct a real analytic map—the Hodge map—defined on a small neighborhood in the deformation space, whose image consists of $(p,p)$-classes on nearby fibers. This analytic description admits an explicit expansion involving higher-order terms in the Beltrami differentials, which leads to further applications in the geometry of compact K\"ahler manifolds.

Now we consider the local analytic family $f:\,\X\to \Delta$ over the polydisk 
$$\Delta=\{z=(z_{1},\cdots,z_{N})\in \C^{N}:\, |z_{i}|<\epsilon,1\le i\le N\},$$
and the corresponding period map of weight $2p$
\begin{equation}\label{pm wt 2p}
\Phi:\,\Delta\to N_{-}\cap D,\, t\mapsto \left( \Phi^{(i,j)}(t) \right)_{0\le i,j\le 2p},
\end{equation}
which maps $t$ to the Hodge structure $H^{2p}(X,\C)=\oplus_{k+l=2p}H^{k,l}(X_{t})$, where $X\cong X_{0}=f^{-1}(0)$.

Under the adapted basis 
$$\eta= \{\eta_{(0)}^T, \cdots, \eta_{(2p)}^T\}^T$$
with respect to the Hodge decomposition $$H^{2p}(X,\C)=\bigoplus_{k+l=2p}H^{k,l}(X),$$ we have the following characterizations of the elements in $H^{2p}(X,\mathbb R)$ and $H^{p,p}(X,\mathbb R)$, where
$$H^{p,p}(X,\mathbb R)=H^{p,p}(X)\cap H^{2p}(X,\mathbb R).$$

\begin{lemma}\label{classes}
For any $\sigma \in H^{2p}(X,\mathbb C)$, there exist row vectors $\alpha_{(0)}, \alpha_{(1)}, \cdots, \alpha_{(2p)}$ such that
$$\sigma = \alpha_{(0)}\cdot \eta_{(0)}+\alpha_{(1)}\cdot \eta_{(1)}+\cdots +\alpha_{(2p)}\cdot \eta_{(2p)}.$$
Moreover we have that

(i) $\sigma \in H^{2p}(X,\mathbb R)$ if and only if $\alpha_{(i)}=\bar{\alpha_{(2p-i)}}$ for $0\le i<p$ and $\alpha_{(p)}=\bar{\alpha_{(p)}}$ is real;

(ii) $\sigma \in H^{p,p}(X,\mathbb R)$ if and only if $\alpha_{(i)}=\bar{\alpha_{(2p-i)}}=0$ for $0\le i<p$, and $\alpha_{(p)}=\bar{\alpha_{(p)}}$ is real.
\end{lemma}
\begin{proof}
(i) Since $\eta$ is the adapted basis with respect to the Hodge decomposition, we have that $$\bar{\eta_{(i)}}=\eta_{(2p-i)},\, 0\le i\le 2p.$$
Hence 
$\sigma \in H^{2p}(X,\mathbb R)$ if and only if
\begin{eqnarray*}
\sigma=\bar{\sigma}=\bar{\alpha_{(0)}}\cdot {\eta_{(2p)}}+\bar{\alpha_{(1)}}\cdot {\eta_{(2p-1)}}+\cdots+\bar{\alpha_{(2p-1)}}\cdot {\eta_{(1)}} + \bar{\alpha_{(2p)}}\cdot {\eta_{(0)}},
\end{eqnarray*}
which is equivalent to that $$\alpha_{(0)}=\bar{\alpha_{(2p)}}, \alpha_{(1)}=\bar{\alpha_{(2p-1)}},\cdots, \alpha_{(p-1)}=\bar{\alpha_{(p+1)}}$$ and $\alpha_{(p)}=\bar{\alpha_{(p)}}$ is real.

(ii) First $\sigma \in H^{p,p}(X,\mathbb R)$ if and only if $\sigma \in H^{2p}(X,\mathbb R)$ and $\sigma \in F^{p}H^{2p}(X,\C)$. Note that 
$$\sigma=\alpha_{(0)}\cdot \eta_{(0)}+\cdots +\alpha_{(p)}\cdot \eta_{(p)}+\cdots +\alpha_{(2p)}\cdot \eta_{(2p)} \in F^{p}H^{2p}(X,\C)$$ 
if and only if $\alpha_{(p+1)}=0,\cdots,\alpha_{(2p)}=0$. Then (ii) follows from (i).
\end{proof}

Let $\sigma_0\in H^{p,p}(X,\mathbb R)$. Then there exists a row vector $\alpha_{(p)}^0$, $\alpha_{(p)}^0=\bar{\alpha}_{(p)}^0$, such that $\sigma_0=\alpha_{(p)}^0 \cdot \eta_{(p)}$. We are interested in the cohomological class $\sigma \in H^{2p}(X,\mathbb R)$ near $\sigma_0$ which is also $(p,p)$-class on $X_{t}$ for $t$ near $0\in \Delta$.

Let $\{\Omega_{(0)}(t), \Omega_{(1)}(t),\cdots,\Omega_{(2p)}(t)\}$ be the adapted basis of $H^{2p}(X,\C)$ with respect to the Hodge filtration at $X_t$, given by the period map \eqref{pm wt 2p}, 
\begin{equation}\label{Omega0-2p}
\Omega_{(i)}(t)=\eta_{(i)}+\sum_{j>i}\Phi^{(i,j)}(t)\cdot \eta_{(j)}.
\end{equation}
 Then, for any real class
$$\sigma = \alpha_{(0)}\cdot \eta_{(0)}+\alpha_{(1)}\cdot \eta_{(1)}+\cdots +\alpha_{(2p)}\cdot \eta_{(2p)}\in H^{2p}(X,\mathbb R),$$
there exist row vectors $\beta_{(0)}(t),\beta_{(1)}(t),\cdots,\beta_{(2p)}(t)$ depending on $t\in \Delta$ such that
\begin{eqnarray}
  \sigma &=& \beta_{(0)}(t)\cdot \Omega_{(0)}(t)+\beta_{(1)}(t)\cdot \Omega_{(1)}(t)+\cdots +\beta_{(2p)}(t)\cdot \Omega_{(2p)}(t).\nonumber
\end{eqnarray}
Therefore, we arrive at the following \textbf{inductive system of equations}:
\begin{equation}\label{Bk}
\left\{
\begin{aligned}
\beta_{(0)}(t)&=\alpha_{(0)},  \\
\beta_{(1)}(t)&=\alpha_{(1)}-\beta_{(0)}(t) \Phi^{(0,1)}(t),  \\
&\ \ \vdots \\
\beta_{(k)}(t)&=\alpha_{(k)} -\sum_{i=0}^{k-1}\beta_{(i)}(t) \Phi^{(i,k)}(t),\\
&\ \ \vdots \\
\beta_{(2p)}(t)&=\alpha_{(2p)}-\sum_{i=0}^{2p-1} \beta_{(i)}(t) \Phi^{(i,2p)}(t),
\end{aligned}\right.
\end{equation}
where each equation uniquely determines $\beta_{(k)}(t)$ from the previously obtained $\beta_{(i)}(t)$ with $i<k$. 
Consequently, all $\beta_{(i)}(t)$ are determined recursively from the fixed data $\alpha_{(j)}$ and the image 
$\Phi(t) = \big(\Phi^{(i,j)}(t)\big)_{i,j}$ of the period map.

\begin{proposition}\label{key pp-proposition}
Any cohomological class
\begin{eqnarray}
 \sigma &=& \alpha_{(0)}\cdot \eta_{(0)}+\alpha_{(1)}\cdot \eta_{(1)}+\cdots +\alpha_{(2p)}\cdot \eta_{(2p)} \nonumber\\
    &=& \beta_{(0)}(t)\cdot \Omega_{(0)}(t)+\beta_{(1)}(t)\cdot \Omega_{(1)}(t)+\cdots +\beta_{(2p)}(t)\cdot \Omega_{(2p)}(t) \label{alpha,beta}
\end{eqnarray}
in $H^{2p}(X,\mathbb C)$ is a real $(p,p)$-class on $X_{t}$, if and only if the following equations hold
\begin{equation}\label{B=01} 
\left\{
\begin{aligned} 
 \beta_{(0)}(t)&=\alpha_{(0)}  \\
\beta_{(1)}(t)&=\alpha_{(1)}-\beta_{(0)}(t) \Phi^{(0,1)}(t)  \\
& \cdots \cdots \\
\beta_{(p)}(t)&=\alpha_{(p)} -\sum_{i=0}^{p-1}\beta_{(i)}(t) \Phi^{(i,p)}(t),\\
\end{aligned}\right.
\end{equation}
\begin{equation}\label{B=02} 
\left\{
\begin{aligned} 
\bar{\alpha_{(p)}} &=\alpha_{(p)}\\
 \bar{\alpha_{(p-1)}} &=\sum_{i=0}^{p} \beta_{(i)}(t)\Phi^{(i,p+1)}(t)\\
  \bar{\alpha_{(p-2)}} &=\sum_{i=0}^{p} \beta_{(i)}(t)\Phi^{(i,p+2)}(t) \\
   & \cdots \cdots \\
  \bar{\alpha_{(0)}}&=\sum_{i=0}^{p} \beta_{(i)}(t)\Phi^{(i,2p)}(t).
\end{aligned}\right.
\end{equation}
\end{proposition}
\begin{proof}
Equation~\eqref{alpha,beta} implies the inductive system~\eqref{Bk}, whose first $p$ equations coincide with~\eqref{B=01}.

Under the identification $H^{2p}(X,\mathbb R)\simeq H^{2p}(X_t,\mathbb R)$, we have that
$\sigma$ is a real $(p,p)$-class on $X_{t}$ if and only if $\sigma\in H^{2p}(X,\mathbb R)\cap  F^p H^{2p}(X_t,\mathbb C)$, which is equivalent to 
\begin{align}
&\alpha_{(i)}=\bar{\alpha_{(2p-i)}},\,0\le i<p;\,\alpha_{(p)}=\bar{\alpha_{(p)}}; \label{real char}\\
&\beta_{(p+1)}(t)=0,\beta_{(p+2)}(t)=0,\beta_{(2p)}(t)=0. \label{Fp char}
\end{align}
Equations \eqref{Fp char} and the last $p$ equations of Equations \eqref{Bk} imply that
\begin{equation}\label{B=01'} 
\left\{
\begin{aligned}   \alpha_{(p+1)}&=\sum_{i=0}^{p} \beta_{(i)}(t)\Phi^{(i,p+1)}(t) \\
  \alpha_{(p+2)}&=\sum_{i=0}^{p+1} \beta_{(i)}(t)\Phi^{(i,p+2)}(t)= \sum_{i=0}^{p} \beta_{(i)}(t)\Phi^{(i,p+2)}(t)\\
   & \cdots \cdots \\
 \alpha_{(2p)} &=\sum_{i=0}^{2p-1} \beta_{(i)}(t)\Phi^{(i,2p)}(t)= \sum_{i=0}^{p} \beta_{(i)}(t)\Phi^{(i,2p)}(t).
\end{aligned}\right.
\end{equation}
Hence Equations \eqref{real char} and \eqref{Fp char} are equivalent to Equations \eqref{B=02}.
\end{proof}

\begin{remark}\label{keypropositionremark}
From Equation \eqref{Bk}, we know that $\beta_{(k)}(t)$ is holomorphic in $t\in \Delta$, and real analytic in $\alpha_{(0)},\cdots,\alpha_{(k)}$ for $0\le k\le p$. In fact, $\beta_{(k)}(t)$ is a linear combination of $\alpha_{(0)},\cdots,\alpha_{(k)}$ with coefficients the holomorphic functions of $t$.
Here we consider $$\alpha_{(i)}=\mathrm{Re}\alpha_{(i)} +\sqrt{-1}\mathrm{Im}\alpha_{(i)},\, \bar{\alpha_{(i)}}=\mathrm{Re}\alpha_{(i)} -\sqrt{-1}\mathrm{Im}\alpha_{(i)}\ 1\le i\le 2p$$ as real variables. 

Therefore, Equations~\eqref{B=02} are characterized as the zeros of certain functions that are real analytic in $\alpha_{(i)},\bar{\alpha_{(i)}}$, for $0 \le i \le p$, and holomorphic in $t \in \Delta$.
\end{remark}

\begin{theorem}\label{real approximation}
Let $f:\, \mathcal X\to \Delta$  be a family of compact K\"ahler manifolds, over a sufficiently small polydisc $\Delta\subset \C^{N}$, where $X\cong X_{0}$. Then there exists a real analytic map
\begin{eqnarray}
H&:&H^{p,p}(X,\mathbb R) \times \Delta\to H^{2p}(X,\mathbb R), \label{defn of H}\\
&&(\sigma,t)\mapsto \sum_{0\le i\le p-1}\alpha_{(i)}(\sigma,t)\cdot \eta_{(i)}+\sigma +\sum_{p+1\le j\le 2p}\bar{\alpha_{(j)}(\sigma,t)}\cdot \eta_{(j)},\nonumber
\end{eqnarray}
such that $$H(\sigma,t)\in H^{p,p}(X_t,\mathbb R).$$
\end{theorem}
\begin{proof}
Let $\sigma=\alpha_{(p)}\cdot \eta_{(p)}\in H^{p,p}(X,\mathbb R)$, and let
\begin{equation}\label{zetaA}
 \zeta=\alpha_{(0)}\cdot \eta_{(0)} +\cdots + \alpha_{(p)}\cdot \eta_{(p)}+\cdots + \alpha_{(2p)}\cdot \eta_{(2p)}
\end{equation}
(not necessarily real) be near $\sigma$ in $H^{2p}(X,\mathbb C)$.

By Proposition \ref{key pp-proposition}, we have that $\zeta\in H^{p,p}(X_t,\mathbb R)$ if and only if Equations \eqref{B=01} and \eqref{B=02} hold.

Note that one can deduce inductively from Equations \eqref{B=01} that
\begin{equation}\label{alg B=01'}
\beta_{(i)}= \alpha_{(i)} -\sum_{l=0}^{i-1}\alpha_{(l)}\sum_{\mu=0}^{i-l-1}(-1)^{\mu}\sum_{\substack{j_{1},\cdots,j_{\mu}\\ l<j_{1}<\cdots<j_{\mu}<i}}\Phi^{(l,j_{1})}(t)\Phi^{(j_{1},j_{2})}(t)\cdots\Phi^{(j_{\mu},i)}(t)
\end{equation}
which is linear in $\alpha_{(0)},\cdots,\alpha_{(i)}$ and holomorphic in $t$, $0\le i\le p$.

Let $$F_{j}=\bar{\alpha_{(p-j)}}-\sum_{i=0}^{p} \beta_{(i)}(t)\Phi^{(i,p+j)}(t),1\le j\le p.$$
Then $F_{j}$ is real analytic in $\alpha_{(0)},\cdots,\alpha_{(p)}, \bar{\alpha_{(p-j)}}$ and holomorphic in $t$, with Jacobians
\begin{eqnarray*}
\frac{\partial F_{j}}{\partial \bar{\alpha_{(p-i)}}}&=&\delta_{ij}I\\
\frac{\partial F_{j}}{\partial {\alpha_{(p-i)}}}&=&-\Phi^{(p-i,p+j)}(t)+\\
&&\sum_{k=p-i+1}^{p}\sum_{\mu=0}^{k-p+i+1}(-1)^{\mu}\sum_{\substack{j_{1},\cdots,j_{\mu}\\ p-i< j_{1}<\cdots<j_{\mu}<k}}\Phi^{(p-i,j_{1})}(t)\cdots\Phi^{(j_{\mu},k)}(t)\Phi^{(k,p+j)}(t).
\end{eqnarray*}
Hence at $t=0$ we have that
$$\frac{\partial F_{j}}{\partial \bar{\alpha_{(p-i)}}}\bigg|_{t=0}=\delta_{ij}I,\, \frac{\partial F_{j}}{\partial {\alpha_{(p-i)}}}\bigg|_{t=0}=O,\, 1\le i,j\le p.$$
Here we use the estimates $\Phi^{(i,p+j)}(t)=O(|t|^{p+j-i})$ in \eqref{lm expansion}.

Therefore the full Jacobian 
$$\mathrm{Jac}(t)=\left(
\begin{array}{cc}
\left(\frac{\partial F_{j}}{\partial \bar{\alpha_{(p-i)}}}\right)_{\substack{1\le j\le p\\1\le i\le p}} &  \left(\frac{\partial F_{j}}{\partial {\alpha_{(p-i)}}}\right)_{\substack{1\le j\le p\\1\le i\le p}} \\
\left(\frac{\partial \bar{F_{j}}}{\partial \bar{\alpha_{(p-i)}}} \right)_{\substack{1\le j\le p\\1\le i\le p}} & \left(\frac{\partial \bar{F_{j}}}{\partial {\alpha_{(p-i)}}}\right)_{\substack{1\le j\le p\\1\le i\le p}} \\
\end{array}\right)$$
depends only on the non-trivial blocks of the period matrices $\left(\Phi^{(p,q)}(t)\right)_{0\le p\le q\le n}$, with initial value 
$\mathrm{Jac}(0)=I$ the identity matrix.

By the Implicit Function Theorem, if the polydisc $\Delta\subset \C^{N}$ is small enough, the equations 
$$F_{j}(\alpha_{(0)}, \ldots, \alpha_{(p)}, t) = 0, \quad 1 \le j \le p,$$
determines functions
$$\alpha_{(i)} = \alpha_{(i)}(\alpha_{(p)}, t),\, t\in \Delta, \quad 0 \le i \le p-1,$$
which are real analytic in the real variables $\bar{\alpha}_{(p)} = \alpha_{(p)}$ and in $t$, and satisfy 
$$\sum_{0\le i\le p-1}\alpha_{(i)}(\alpha_{(p)}, t)\cdot \eta_{(i)}+\alpha_{(p)} \cdot \eta_{(p)} +\sum_{p+1\le j\le 2p}\bar{\alpha_{(j)}(\alpha_{(p)}, t)}\cdot \eta_{(j)}\in H^{p,p}(X_t,\mathbb R)$$
with initial data $\alpha_{(i)}(\alpha_{(p)}^0, 0) = 0$.

Therefore, the map $H$ in \eqref{defn of H} is well-defined with the functions $\alpha_{(i)}(\sigma, t) = \alpha_{(i)}(\alpha_{(p)}, t)$ for $\sigma = \alpha_{(p)} \cdot \eta_{(p)}$.
\end{proof}

\begin{definition}
The map $H$ in \eqref{defn of H} is called Hodge map.
\end{definition}

\begin{remark}\label{real approximation remark}
(1) The smooth variation of the cohomology groups $H^{p,q}(X_t)$ for $t \in \Delta$ can also be deduced from the classical results of Kodaira and Spencer~\cite{KS3}. By contrast, our construction for $(p,p)$-classes provides an explicit expansion involving all higher-order terms of the Beltrami differential, and the associated Hodge map $H$ in \eqref{defn of H} is defined on a uniform polydisc $\Delta$ for all classes in $H^{p,p}(X,\mathbb{R})$. This uniformity plays a key role in our applications to the geometry of compact K\"ahler manifolds.

(2) Note that Implicit Function Theorem also holds for functions
$$F_{j}(\alpha_{(0)}, \ldots, \alpha_{(p)}, t), \quad 1 \le j \le p,$$
when $t$ are variables in the singular analytic space $B$, e.g. the Kuranishi space which will be introduced below. Hence Theorem \ref{real approximation} also holds for the Kuranishi space $B$.
\end{remark} 


\section{K\"ahler stability and upper semicontinuity of K\"ahler cones}\label{Kahler section}
This section is devoted to two complementary descriptions of the variation of K\"ahler cones under deformation. First, using the reference Hodge connection $\nabla_{t_0}^{1,1}$, we prove a uniform upper semicontinuity property: the $\nabla_{t_0}^{1,1}$-flat extension of the K\"ahler cone of a reference fiber is contained in the K\"ahler cones of all sufficiently nearby fibers. This also yields a proof of K\"ahler stability that does not require unobstructedness, namely that sufficiently small deformations of a compact K\"ahler manifold remain K\"ahler.

Second, using the moving Hodge connection $\nabla^{1,1}$ introduced by Demailly--Paun, we describe the parallel transport of K\"ahler cones on the locus where analytic cycles deform and give an explicit formula for the corresponding flat extensions in terms of the period map. In this way, we refine and extend Theorem~0.9 of \cite{DP04} by distinguishing the uniform inclusion furnished by $\nabla_{t_0}^{1,1}$ from the identification of K\"ahler cones furnished by $\nabla^{1,1}$.

In \cite{KS3}, Kodaira and Spencer established the stability of the K\"ahler property under small deformations of complex structures using the theory of families of elliptic operators. 
Later, in Remark~1 on page~180 of \cite{MorrowKodaira}, Kodaira suggested that a more elementary argument based on power series methods should also be possible. 
Recently, Rao, Wan, and Zhao~\cite{RaoWanZhao22} provided a direct proof in response to Kodaira’s question.

In this section, we construct the quasi-period map of weight two in \eqref{qpm construction}. 
This construction leads first to a proof of K\"ahler stability without assuming the smoothness of the deformation base, which serves mainly as a warm-up for the subsequent analysis. 
More importantly, it allows us to establish the upper semicontinuity of K\"ahler cones in families, and the argument naturally extends beyond the local setting to deformation families over comparatively large regions.

First we review the Kuranishi's theorem of existence of deformation of compact complex manifolds over analytic base. One can refer to Chapter 4.3 of \cite{MorrowKodaira} for details.

Let $X$ be a compact complex manifold. Let $\bp,\bp^{*}, \triangle_{\bp}=\bp \bp^*+\bp^*\bp$ be the standard operators from complex geometry acting on $A^{p,q}(X,\Tan X)$. Let $G$ be the Green's operator such that 
$$I_{A^{p,q}(X,\Tan X)}=\mathbb H + G\triangle_{\bp}=\mathbb H + \triangle_{\bp}G.$$
We fix a basis $\theta_{1},\dots,\theta_{N}$ of $\mathbb{H}^{0,1}(X,\Tan X)$. 
Consider the smooth section $\phi(t) \in A^{0,1}(X,\Tan X)$ of the form
$$
\phi(t)=\phi_{1}(t)+\phi_{2}(t)+\cdots+\phi_{\mu}(t)+\cdots, \quad t\in \Delta \subset \C^{N},
$$
where each homogeneous term is given by
$$
\phi_{\mu}(t)=\sum_{i_{1}+\cdots+i_{N}=\mu}
\phi_{i_{1}\cdots i_{N}}\,t_{1}^{i_{1}}\cdots t_{N}^{i_{N}},
$$
and the $\phi_{\mu}(t)$ are determined inductively by
\begin{eqnarray}
\phi_{1}(t)&=&\sum_{i}\theta_{i}t_{i},\nonumber\\
\phi_{\mu}(t)&=&\frac{1}{2}\bp^{*}G\left(\sum_{1\le \nu\le \mu-1}
\left[\phi_{\nu}(t),\phi_{\mu-\nu}(t)\right]\right).
\end{eqnarray}
Equivalently,
\begin{equation}\label{Kuranishi phi}
\phi(t)=\sum_{i}\theta_{i}t_{i}+\frac{1}{2}\bar{\partial}^{*}G[\phi(t),\phi(t)].
\end{equation}
For $\Delta$ of sufficiently small polyradius $\epsilon$, the section $\phi(t)$ depends holomorphically on $t\in \Delta$.

\begin{theorem}[Kuranishi]\label{Kuranishi}
\

(1) Let $X$ be a compact complex manifold, and $\phi(t)\in A^{0,1}(X,\Tan X)$ be defined as \eqref{Kuranishi phi}. 
Define the analytic space $B$ containing $0$ by
 $$B = \left\{ t \in \Delta:\, \mathbb H[\phi(t), \phi(t)] = 0 \right\}.$$
  Then for each $t \in B$, the Beltrami differential $\phi(t)$ defines a complex structure $X_t$ on the differentiable manifold $X$.

(2) Let $\psi$ be any Beltrami differential, i.e. $\psi\in A^{0,1}(X,\Tan X)$ satisfying the integrability condition
  $$\bar{\partial} \psi - \frac{1}{2}[\psi, \psi] = 0.$$
  Then $\psi$ defines a complex structure $X_{\psi}$ on $X$. If $\|\psi\|_{k}$ is sufficiently small in a suitable Sobolev norm (Page 159 of \cite{MorrowKodaira}), then there exists a unique $a \in \mathcal{F}^0$ (the group of smooth diffeomorphisms isotopic to the identity) such that $a^* \psi = \phi(t)$ for some $t \in B$. Hence $X_{\psi}$ is isomorphic to $X_t=X_{\phi(t)}$ as complex manifolds.
\end{theorem}

The analytic family $\{X_{t}:\, t\in B\}$ is called Kuranishi family with Kuranishi base $B$.

Now we assume that $X$ is compact K\"ahler. Let $\eta_{(0)},\eta_{(1)},\eta_{(2)}$ be the adapted basis of the Hodge decomposition 
$$H^{2}(X,\C)=H^{2,0}(X)\oplus H^{1,1}(X)\oplus H^{0,2}(X),$$
where $\eta_{(i)}=[\tilde\eta_{(i)}]$ with $\tilde\eta_{(i)}$ the basis of $\mathbb H^{2-i,i}(X)$, $0\le i\le 2$.
Let $\Omega_{(i)}(t)$, $0\le i\le 2$ be defined by 
\begin{equation}\label{qpm construction}
\Omega_{(i)}(t)=\left[\mathbb H \left(e^{i_{{\phi(t)}}}\left((I+Ti_{\phi(t)})^{-1}\tilde\eta_{(i)}\right)\right)\right],\, t\in \Delta.
\end{equation}
Explicitly, we have that
\begin{eqnarray*}
\Omega_{(0)}(t)&=& \eta_{(0)} + \left[\mathbb H \left(i_{\phi(t)}\left((I+Ti_{\phi(t)})^{-1}\tilde\eta_{(0)}\right)\right)\right]+\left[\mathbb H \left(i_{\phi(t)}^{2}\left((I+Ti_{\phi(t)})^{-1}\tilde\eta_{(0)}\right)\right)\right]\\
\Omega_{(1)}(t)&=& \eta_{(1)} + \left[\mathbb H \left(i_{\phi(t)}\left((I+Ti_{\phi(t)})^{-1}\tilde\eta_{(1)}\right)\right)\right]\\
\Omega_{(2)}(t)&=& \eta_{(2)}.
\end{eqnarray*}
By the definition of $\phi(t)$ in \eqref{Kuranishi phi}, we have that $\Omega_{(i)}(t)$ is holomorphic in $t\in \Delta$, $0\le i\le 2$. Moreover there exist matrices $\Phi^{0,1}(t),\Phi^{0,2}(t),\Phi^{1,2}(t)$ such that
\begin{eqnarray*}
\Omega_{(0)}(t)&=& \eta_{(0)} + \Phi^{0,1}(t)\cdot\eta_{(1)}+ \Phi^{0,2}(t)\cdot\eta_{(2)}\\
\Omega_{(1)}(t)&=& \eta_{(1)} + \Phi^{1,2}(t)\cdot\eta_{(2)}\\
\Omega_{(2)}(t)&=& \eta_{(2)}.
\end{eqnarray*}
In fact, 
$$ \Phi^{i,j}(t)=\left(\mathbb H \left(i_{\phi(t)}^{j-i}\left((I+Ti_{\phi(t)})^{-1}\tilde\eta_{(i)}\right)\right),\tilde\eta_{(j)}\right),\, 0\le i<j\le 2,$$
which are holomorphic in $t$, where $(\cdot,\cdot)$ is the Hermitian form on $A^{2}(X)$.

The quasi-period map $\Phi:\,\Delta \to N_{-}\cap D$ is given by mapping $t$ to the matrix 
$$\left(
\begin{array}{ccc}
I & \Phi^{0,1}(t) & \Phi^{0,2}(t) \\
O& I& \Phi^{1,2}(t)\\
O&O&I
\end{array}\right),
$$which takes values in $D$ for the radius $\epsilon$ of $\Delta$ sufficiently small, where $D$ denotes the classifying space of weight-two (non-polarized) Hodge structures on $H^{2}(X, \mathbb{C})$.

We call the above ``quasi-period'' map because it represents the usual periods of integrals on $X$ only when $t\in B$ in the sense of Griffiths \cite{Griffiths1,Griffiths2}. Precisely, when $t\in B$, $\phi(t)$ represents an integrable complex structure on $X$, which defines a filtration $F^{\bullet}$ on $A^{2}(X_{t})$ such that
$$e^{i_\phi(t)}: F^pA^{2}(X_{t})\rightarrow F^p A^{2}(X_{t}),$$
is an isomorphism for $0\le p\le 2$. Moreover, from the results in Section \ref{SHB}, we have that 
\begin{equation}\label{section weight 2}
e^{i_{{\phi(t)}}}\left((I+Ti_{\phi(t)})^{-1}\tilde\eta_{(i)}\right)\in F^{2-i} A^{2}(X_{t})
\end{equation}
is $d$-closed, and hence its cohomological classes $[\cdots]=[\mathbb H\cdots]$ lies in
$F^{2-i} H^{2}(X_{t},\C)$, for $0\le i\le 2$.
Therefore we have the following well-defined period map
\begin{equation}\label{qpm 2 defn}
\Phi:\, B\to  N_{-}\cap D,\, t\mapsto \left(\Phi^{(p,q)}(t)\right)_{0\le p\le q\le 2}.
\end{equation}
\begin{proposition}\label{qpm 2}
For the period map \eqref{qpm 2 defn}, we have the Griffiths transversality 
\begin{equation}\label{trans 2}
\left(\Phi^{0,2}\right)_{\mu}^{\bullet}(t)= \left(\Phi^{0,1}\right)_{\mu}^{\bullet}(t)\Phi^{1,2}(t),
\end{equation}
where $(\cdot)_{\mu}^{\bullet}=\frac{\partial}{\partial t_{\mu}}$ is the tangent vector in the Zariski tangent space of $B$ at $t$,
$$\mathrm T_{t}B =\left\{\frac{\partial}{\partial t_{\mu}}\in T_{t}\Delta:\, \frac{\partial}{\partial t_{\mu}}f=0,\,\forall\, f\in \mathcal I(B)\right\}.$$
Hence we have the estimates
 $$\P^{(0,1)}(t),\P^{(1,2)}(t)=O(|t|),\P^{(0,2)}(0)=O(|t|^{2}),\,t\in B$$
of the blocks $\Phi^{i,j}(t)$ as in Lemma \ref{order0}.
\end{proposition}
\begin{proof}
Since the classes in $F_t^{2-i} H^2(X_t, \mathbb{C})$ have already been constructed in~\eqref{section weight 2}, they can be written locally as
$$
\sum_{|I| \le 2-i} f_{IJ}(t) \left( dz_I + e^{i_{\phi(t)}} dz_I \right) \wedge \left( d\bar{z}_J + \overline{e^{i_{\phi(t)}} dz_J} \right).
$$
Applying the proof of Proposition~(1.11) in~\cite{Griffiths2}—Griffiths’ original argument for the transversality for the period map—we obtain the Griffiths transversality for our quasi-period map:
$$
\left( \Omega_{(0)} \right)_{\mu}^{\bullet}(t) = \left( \Phi^{0,1} \right)_{\mu}^{\bullet}(t) \cdot \Omega_{(1)}(t),
$$
which implies \eqref{trans 2}.
\end{proof}

As an application of the period map \eqref{qpm 2 defn} and the corresponding Hodge map introduced in Section~\ref{pp Section}, we establish the following theorem, which serves as a warm-up for the proof of Theorem~\ref{Kahler cone invariance} below.

\begin{theorem}\label{Kahler stability}
Let $(X,\omega_{0})$ be a compact K\"ahler manifold whose Kuranishi base $B$ has positive dimension. Then all the complex manifolds $X_{\psi}$ sufficiently near $X$ as given in (2) of Theorem \ref{Kuranishi} are also K\"ahler.
\end{theorem}
\begin{proof}
From Kuranishi's theorem, Theorem \ref{Kuranishi}, we only need to show that $X_{t}$ is K\"ahler for $t\in B$ and the radius $\epsilon$ of $B$ sufficiently small.

Since the image of $\Phi$ lies in $D$, we have the Hodge structure 
$$H^{2}(X,\C)=H^{2,0}(X_{t})\oplus H^{1,1}(X_{t})\oplus H^{0,2}(X_{t}),$$
for $t\in B$, where the Hodge structure is induced by the Hodge filtration $F^{\bullet}$ on $A^{2}(X_{t})$ introduced as above.

Let $\sigma_{0}=[\omega_{0}]=\alpha_{(1)}^{0}\cdot \eta_{(1)}\in H^{1,1}(X,\mathbb R)$ be the K\"ahler class. 
Applying the proof of Theorem \ref{real approximation} to the period map 
$$\Phi:\,B \to N_{-}\cap D,$$ there exist a real analytic Hodge map 
$$H:\, H^{1,1}(X,\mathbb R) \times B\to \bigcup_{t\in B}H^{1,1}(X_t,\mathbb R).$$
See also Remark \ref{real approximation remark}.
Hence $$H(\sigma_{0},t)=\alpha_{(0)}(\sigma_{0},t)\cdot \eta_{(0)}+\sigma_{0}+\bar{\alpha_{(0)}(\sigma_{0},t)}\cdot \eta_{(2)}\in H^{1,1}(X_t,\mathbb R)$$ is a real analytic extension of $\sigma_{0}$ over $t\in B$, where the real analytic function $\alpha_{(0)}(\sigma_{0},t)$ is determined by
\begin{equation}\label{B=01 11-0}
\bar{\alpha_{(0)}}- \alpha_{(1)}^{0}\Phi^{(1,2)}(t)-\alpha_{(0)}\left(\Phi^{(0,2)}(t)-\Phi^{(0,1)}(t)\Phi^{(1,2)}(t)\right)=0.
\end{equation}

Explicitly, we have the harmonic representative of $H(\sigma_{0},t)$,
\begin{eqnarray}\label{rep of 11class 1}
\tilde H(\sigma_{0},t)&=&\alpha_{(0)}(\sigma_{0},t)\cdot \tilde\eta_{(0)}+\omega_{0}+\bar{\alpha_{(0)}(\sigma_{0},t)}\cdot \tilde\eta_{(2)}.
\end{eqnarray}
Note that, in our setting, the harmonic representative of $H(\sigma_{0},t)$ is taken with respect to the base fiber $X$, in contrast to the approach in the proof of Theorem~4.6 of~\cite{MorrowKodaira}.

From \eqref{basis of cotang} and \eqref{basis of cotang bar}, we see that the harmonic representative can be locally written by
$$\tilde H(\sigma_{0},t)=\sum_{ij}g_{i\bar j}(z,t)\left(dz_{i}+\phi(t)dz_{i}\right)\wedge\left(d\bar{z_{j}}+\bar{\phi(t)}d\bar{z_{j}}\right),$$
where $g_{i\bar j}(z,t)$ is real analytic in $t$, as both the row vector $\alpha_{(0)}(\sigma_{0},t)$ and the bases of ${\mathrm{T}^*}^{1,0}X_t$, ${\mathrm{T}^*}^{0,1}X_t$ depend real analytically on $t$.

Since 
$$\tilde H(\sigma_{0},0)=\omega_{0}=\sum_{ij}g_{i\bar j}(z,0)\,dz_{i}\wedge d\bar{z}_{j}$$ 
is positive definite, there exists $\epsilon_{0}>0$ such that the Hermitian matrix $\left(g_{i\bar j}(z,t)\right)$ remains positive definite for $|t|<\epsilon_{0}$. Consequently, the harmonic representative $\tilde H(\sigma_{0},t)$ is positive definite for all $t\in B$ provided that the radius of $B$ is less than $\epsilon_{0}$. In particular, $(X_{t}, H(\sigma_{0},t))$ is K\"ahler for every $t\in B$.
\end{proof}

Now we consider an analytic family $\mathcal X \to S$ of compact K\"ahler manifolds over an irreducible analytic base $S$.
Recall that we have defined the $\nabla_{t_{0}}^{1,1}$-flat extensions $\mathcal{K}^{\nabla_{t_0}^{1,1}}_{t_{0},t}$ of the K\"ahler cone $\mathcal{K}_{t_{0}}$ for $t_{0}\in S$ in Definition \ref{extensions of the Kahler cone defn} by
$$
\mathcal{K}^{\nabla_{t_0}^{1,1}}_{t_{0},t}
=
\left\{ H(\sigma,t)=\alpha_{(0)}(\sigma,t)\cdot \eta_{(0)}+\sigma+\bar{\alpha_{(0)}(\sigma,t)}\cdot \eta_{(2)}:\, \sigma\in \mathcal{K}_{t_{0}}\right\}
.
$$
Then we can prove that the positivity of the harmonic representatives $\tilde H(\sigma,t)$ of the elements $H(\sigma,t)$, with respect to the K\"ahler form representing $\sigma\in \mathcal{K}_{t_{0}}$, depends only on the Beltrami differential $\phi(t)$.
This gives the extensions of the K\"ahler cone $\mathcal{K}_{t_{0}}$, which strengthens Theorem~0.9 in \cite{DP04} by Demailly--Paun.

\begin{theorem}[Upper semicontinuity of K\"ahler cones]\label{Kahler cone invariance}
Let $\mathcal X \to S$ be a deformation of compact K\"ahler manifolds over an irreducible analytic base $S$. Then the section $t\mapsto \mathcal{K}_{t} \subset H^{1,1}(X_{t},\mathbb{C})$ of K\"ahler cones
is upper semicontinuity under parallel transport with respect to the $(1,1)$-projection $\nabla_{t_0}^{1,1}$ of the Gauss–Manin connection (see Section 5 of \cite{DP04}). Precisely, for any $t_{0}\in S$, there exists a sufficiently small neighborhood $U$ of $t_{0}$, such that the $\nabla_{t_0}^{1,1}$-flat extensions $\mathcal{K}^{\nabla_{t_0}^{1,1}}_{t_{0},t}$ of the K\"ahler cone $\mathcal{K}_{t_{0}}$ exist
and satisfy that
\begin{equation}\label{local of K cone}
\mathcal{K}^{\nabla_{t_0}^{1,1}}_{t_{0},t}\subset \mathcal{K}_{t}.
\end{equation}
for all $t\in U$.
\end{theorem}
\begin{proof} 
We only need to prove \eqref{local of K cone}. Since the statement is local and the Kuranishi family is complete, it is enough to work on the Kuranishi base $B$.

Let $(X_{0},\omega_{0})$ be a compact K\"ahler manifold which is identified with the fiber $X_{0}$ over $0\in B$. Then we have the Beltrami differential $\phi(t)$ and the period map $\Phi:\, B \to N_{-} \cap D$ as before, using the harmonic theory with respect to the fixed initial K\"ahler form $\omega_{0}$.

Let $\sigma=[\omega]\in \mathcal{K}_{t_{0}}$ be any K\"ahler class, where the corresponding K\"ahler form
\begin{equation}\label{omega at 0}
\omega=\sum_{ij}g_{i\bar j}(z)dz_{i}\wedge d\bar{z_{j}}
\end{equation}
is a positive definite $(1,1)$-form on $X$. 
Let $\tilde \eta_{(i)}^{\omega}$ be the harmonic representatives of $\eta_{(i)}$ with respect to the Laplacian $\triangle^{\omega}=\triangle^{\omega}_{d}=1/2\triangle^{\omega}_{\bar\partial}$ induced by the K\"ahler form $\omega$. Then $H(\sigma,t)\in H^{1,1}(X_{t},\mathbb R)$ has the harmonic representative
\begin{equation}\label{rep of 11class 1'}
\tilde H^{\omega}(\sigma,t)=\alpha_{(0)}(\sigma,t)\cdot\tilde \eta^{\omega}_{(0)}+\omega+\bar{\alpha_{(0)}(\sigma,t)}\cdot \tilde \eta^{\omega}_{(2)}.
\end{equation}
Note that the real analytic function $\alpha_{(0)}(\sigma,t)$ is determined by the equation \eqref{B=01 11-0}, where the blocks $\Phi^{(p,q)}(t)$ are obtained from 
$$e^{i_{\phi(t)}}\left((I+Ti_{\phi(t)})^{-1}\tilde\eta_{(p)}\right),$$ 
and the Beltrami differential $\phi(t)$ and the operator $T=\bar{\partial}^{*}G\partial$ are defined with respect to the fixed K\"ahler form $\omega_{0}$.

From \eqref{basis of cotang} and \eqref{basis of cotang bar} again, we have that 
\begin{eqnarray}
\tilde H^{\omega}(\sigma,t)&=&\sum_{ij}g_{i\bar j}(z,t)\left(dz_{i}+\phi(t)dz_{i}\right)\wedge\left(d\bar{z_{j}}+\bar{\phi(t)}d\bar{z_{j}}\right)\label{expl rep of 11class'}\\
&=& \sum_{ij}g_{i\bar j}(z,t)dz_{i}\wedge( \bar{\phi(t)} d\bar{z_{j}}) +\sum_{ij}\left(g_{i\bar j}(z,t) \right.\label{rep of 11class 2'}\\
&&\left.-\phi_{\bar j}^{k}(z,t)\bar{\phi_{\bar i}^{l}(z,t)}g_{k\bar l}(z,t)\right)dz_{i}\wedge d\bar{z_{j}}+\sum_{ij}g_{i\bar j}(z,t)({\phi(t)}dz_{i})\wedge d\bar{z_{j}},\nonumber
\end{eqnarray}
where $$\phi(t)=\sum_{ij}\phi_{\bar j}^{i}(z,t){\partial_{i}}\otimes d\bar{z_{j}}.$$
Since the Laplacian operator $\Delta^{\omega}$ preserves $(p,q)$-types and $\omega$ is harmonic with respect to $\Delta^{\omega}$, a comparison of types in \eqref{rep of 11class 1'} and \eqref{rep of 11class 2'} shows that
\begin{equation}\label{omega at t}
\omega=\sum_{ij}\left(g_{i\bar j}(z,t)-\phi_{\bar j}^{k}(t)g_{k\bar l}(z,t)\bar{\phi_{\bar i}^{l}(t)}\right)dz_{i}\wedge d\bar{z_{j}}.
\end{equation}
Therefore, from \eqref{omega at 0} and \eqref{omega at t}, it follows that
$$\left(g_{i\bar j}(z,t)\right)=\left(I-S(z,t)\right)^{-1}\left(g_{i\bar j}(z)\right).$$
Here $S(z,t):\, M_{d}(\C)\to M_{d}(\C)$ denotes the operator 
$$G \mapsto \bar{\phi(z,t)}^{T}\, G^{T}\, {\phi(z,t)},$$
depending real analytically on $z$ and $t$, where $\phi(z,t) = \left(\phi_{\bar j}^{k}(z,t)\right)_{1\le k,j\le d}$ is regarded as a matrix. It is straightforward to verify that $S(z,t)$ maps Hermitian matrices to Hermitian matrices, and that the inverse $\left(I-S(z,t)\right)^{-1}\left(g_{i\bar j}(z)\right)=\left(g_{i\bar j}(z)\right)+\sum_{k\ge 1}S(z,t)^{k}\left(g_{i\bar j}(z)\right)$ exists whenever the norm
$$\|\phi\|^{E}=\max_{U}\max_{z\in \bar{V}} \sigma_{\mathrm{max}}\left(\phi_{\bar j}^{k}(z)\right)<1,$$
where 
\begin{equation}\label{fc of X}
X = \bigcup_{\text{finite}} (U; (z_{1},\dots,z_{d})),
\end{equation}
is a finite cover, and for each chart $U$ there exist subsets $V \subset \overline{V} \subset U$ such that $\bigcup V$ still covers $X$. See \eqref{son E} for details.

Moreover, if $G$ is positive definite, then so is $\left(I-S(t)\right)^{-1}(G)$, since for any nonzero column vector $v$ we have
\begin{eqnarray*}
v^{T}\left(I-S(t)\right)^{-1}(G)\bar{v}
  &=& v^{T}G\bar{v}+\sum_{k\ge 1}\left(\bar{\phi(t)}^{k}v\right)^{T}G^{T}\overline{\left(\bar{\phi(t)}^{k}v\right)} \\
  &\ge & v^{T}G\bar{v}>0.
\end{eqnarray*}

Since $\phi(0)=0$, there exists $\epsilon_{0}>0$ such that $\|\phi(t)\|^{E}<1$ whenever $|t|<\epsilon_{0}$. 
Then the above argument implies that the Hermitian matrices $\left(g_{i\bar j}(z,t)\right)$ are positive definite 
for all $|t|<\epsilon_{0}$. Here $\epsilon_{0}$ depends only on the Beltrami differential $\phi(t)$ and the fixed finite cover \eqref{fc of X} of $X$.

From \eqref{expl rep of 11class'}, it follows that $H(\sigma,t)$ admits a representative $\tilde H^{\omega}(\sigma,t)$ which is positive definite on $X_{t}$, for any $t\in B$, provided that the radius of $B$ is less than $\epsilon_{0}$. Consequently, $H(\sigma,t)\in \mathcal{K}_{t}$ for all $t\in B$ and every $\sigma\in \mathcal{K}_{t_{0}}$.


Finally we finish the proof of the theorem.
\end{proof}

\begin{remark}\label{app to nef cones}
Taking closures in \eqref{local of K cone}, we obtain the corresponding upper semicontinuity of the nef cones:
$$\overline{\mathcal{K}^{\nabla_{t_0}^{1,1}}_{t_0,t}}
\subset
\overline{\mathcal K_t}
=
\operatorname{Nef}(X_t).$$
Equivalently, the $\nabla_{t_0}^{1,1}$-flat extension of $\operatorname{Nef}(X_{t_0})$ is contained in $\operatorname{Nef}(X_t)$ for every $t\in U$.

If $\mathcal X\to S$ is a family of projective manifolds, then, by the duality between the nef cone and the closed cone of effective curves, the induced dual parallel transport gives the reversed inclusion
$$\overline{\operatorname{NE}}(X_t)
\subset
\left(
\operatorname{Nef}^{\nabla_{t_0}^{1,1}}_{t_0,t}
\right)^{\vee}.$$
Thus, under the corresponding identifications, the cone of effective curves cannot enlarge in a sufficiently small deformation. This provides constraints on the variation of extremal faces, extremal curve classes, and the associated contraction morphisms, and may therefore be useful in studying the deformation theory of birational structures.
\end{remark}

In Theorem~0.9 of \cite{DP04}, Demailly and Paun proved the invariance of the K\"ahler cone, under the $\nabla^{1,1}$-flat extensions, on the complement $S \setminus S'$ of a countable union $S' = \bigcup_{\nu} S_{\nu}$ of analytic subsets $S_{\nu} \subset S$.
We now prove the explicit formula for the  $\nabla^{1,1}$-flat extensions as introduced in Proposition~\ref{intr DP Kahler extensions}. 

\begin{theorem}\label{explicit DP Kahler extensions}
Let $t_0\in S$. There exists an endomorphism-valued function
$$M(t)\in\operatorname{End}\bigl(H^{1,1}(X_{t_0},\mathbb R)\bigr),$$
uniquely determined by the blocks $\Phi^{(0,1)}(t)$, $\Phi^{(0,2)}(t)$, and $\Phi^{(1,2)}(t)$ of the period map, such that, for every $\sigma_{t_0}\in H^{1,1}(X_{t_0},\mathbb R)$, the corresponding $\nabla^{1,1}$-flat extension is given by
$$H(\sigma_t,t),$$
where
\begin{equation}\label{formula of Kahler extensions}
\sigma_t=\exp\left(\int_{\gamma_{t_0,t}}M(s)\,ds\right)\sigma_{t_0}.
\end{equation}
Here $\gamma_{t_0,t}$ is a path in $S$ joining $t_0$ to $t$, and the resulting extension may depend on the choice of $\gamma_{t_0,t}$.

Consequently,
$$\mathcal K^{\nabla^{1,1}}_{t_0,t}
=
\left\{
H\left(
\exp\left(\int_{\gamma_{t_0,t}}M(s)\,ds\right)\sigma_{t_0},
t
\right):
\sigma_{t_0}\in\mathcal K_{t_0}
\right\}.$$
Moreover, for the countable union $S'\subset S$ of analytic subsets appearing in Demailly--Paun \cite{DP04}, one has
$$\mathcal K^{\nabla^{1,1}}_{t_0,t}=\mathcal K_t$$
for $t_0\in S\setminus S'$ and $t$ sufficiently close to $t_0$ with $t\in S\setminus S'$.
\end{theorem}
\begin{proof}
We express the moving Hodge decomposition in terms of the graph maps determined by the period map and then derive the corresponding matrix differential equation.

Fix the real decomposition
$$H^2(X_{t_0},\mathbb R)=V\oplus W,$$
where
$$V:=H^{1,1}(X_{t_0},\mathbb R),\qquad
W:=\bigl(H^{2,0}(X_{t_0})\oplus H^{0,2}(X_{t_0})\bigr)\cap H^2(X_{t_0},\mathbb R).$$
For $t$ sufficiently close to $t_0$, the Hodge map
$$H(\,\cdot\,,t):V\longrightarrow H^{1,1}(X_t,\mathbb R)$$
determines a linear map
$$\Gamma_t:V\longrightarrow W,\qquad \Gamma_{t_0}=0,$$
such that
$$H(v,t)=v+\Gamma_t(v),\qquad v\in V.$$
Consequently, the moving real $(1,1)$-subspace is precisely the graph of $\Gamma_t$:
$$H^{1,1}(X_t,\mathbb R)=\operatorname{Graph}(\Gamma_t)
=\left\{H(v,t)=v+\Gamma_t(v):\,v\in V\right\}.$$

Similarly, the explicit expression
$$\Omega_{(0)}(t)
=
\eta_{(0)}
+\Phi^{(0,1)}(t)\cdot\eta_{(1)}
+\Phi^{(0,2)}(t)\cdot\eta_{(2)}
\in F^2H^2(X_t,\mathbb C)$$
determines the moving subspace $F^2H^2(X_t,\mathbb C)=H^{2,0}(X_t)$, and hence, together with its complex conjugate, determines the real subspace underlying
$$H^{2,0}(X_t)\oplus H^{0,2}(X_t).$$
Therefore, there exists a uniquely determined linear map
$$\Delta_t:W\longrightarrow V,\qquad \Delta_{t_0}=0,$$
such that
$$\bigl(H^{2,0}(X_t)\oplus H^{0,2}(X_t)\bigr)\cap H^2(X_{t_0},\mathbb R)
=
\operatorname{Graph}(\Delta_t)
=
\left\{\Delta_t(w)+w:\,w\in W\right\}.$$

Thus the maps $\Gamma_t$ and $\Delta_t$ are uniquely determined by the blocks $\Phi^{(0,1)}(t)$, $\Phi^{(0,2)}(t)$, and $\Phi^{(1,2)}(t)$ of the period map, and
$$H^2(X_{t_0},\mathbb R)
=
\operatorname{Graph}(\Gamma_t)
\oplus
\operatorname{Graph}(\Delta_t).$$

For $v\in V$ and $w\in W$, write
$$v+w=(a+\Gamma_ta)+(\Delta_tb+b),$$
where $a\in V$ and $b\in W$. Comparing the $V$- and $W$-components gives
$$v=a+\Delta_tb,\qquad w=\Gamma_ta+b.$$
Eliminating $b$, we obtain
$$(I-\Delta_t\Gamma_t)a=v-\Delta_tw,$$
and therefore
$$a=(I-\Delta_t\Gamma_t)^{-1}(v-\Delta_tw).$$
It follows that the projection onto the moving real $(1,1)$-subspace is
$$P_t(v+w)
=
H\left((I-\Delta_t\Gamma_t)^{-1}(v-\Delta_tw),t\right).$$
Equivalently,
$$H(\,\cdot\,,t)^{-1}P_t(v+w)
=
(I-\Delta_t\Gamma_t)^{-1}(v-\Delta_tw).$$

Let
$$\alpha_t=H(\sigma_t,t)=\sigma_t+\Gamma_t(\sigma_t),\qquad \sigma_t\in V.$$
The condition that $\alpha_t$ is $\nabla^{1,1}$-flat is
$$P_t\frac{d\alpha_t}{dt}=0.$$
Since $H(\,\cdot\,,t)$ is linear in its first variable,
$$\frac{d\alpha_t}{dt}
=
H(\dot\sigma_t,t)+\dot\Gamma_t(\sigma_t).$$
Here $H(\dot\sigma_t,t)\in H^{1,1}(X_t,\mathbb R)$, whereas
$\dot\Gamma_t(\sigma_t)\in W$. Consequently,
$$H(\dot\sigma_t,t)+P_t\bigl(\dot\Gamma_t(\sigma_t)\bigr)=0.$$
Applying $H(\,\cdot\,,t)^{-1}$ and using the projection formula with $v=0$ and
$w=\dot\Gamma_t(\sigma_t)$ gives
$$\dot\sigma_t
=
(I-\Delta_t\Gamma_t)^{-1}\Delta_t\dot\Gamma_t\,\sigma_t.$$

Therefore, the coefficient endomorphism in Proposition~\ref{intr DP Kahler extensions} is explicitly given by
$$M(t)=(I-\Delta_t\Gamma_t)^{-1}\Delta_t\dot\Gamma_t.$$
The equation
$$\dot\sigma_t=M(t)\sigma_t,\qquad \sigma_{t_0}\in H^{1,1}(X_{t_0},\mathbb R),$$
is a linear matrix differential equation along the chosen path
$\gamma_{t_0,t}$. Its solution is
$$\sigma_t=\exp\left(\int_{\gamma_{t_0,t}}M(s)\,ds\right)\sigma_{t_0},$$
which proves formula~\eqref{formula of Kahler extensions}.
\end{proof}
\section{K\"ahler stability on large scales}\label{large KS}
In this section, we study a large-scale form of K\"ahler stability based on the extension of $(1,1)$-classes via period matrices. We show that both the existence and the positivity of the extended classes are determined by the Beltrami differentials, leading to a uniform criterion for K\"ahlerness over a large region in the base. As a result, we obtain a global K\"ahler stability theorem and the existence of $\nabla_{t_0}^{1,1}$-flat extensions of K\"ahler cones on this region.

Let $f:\,\X \to S$ be an analytic family of compact complex manifolds over a connected complex manifold $S$, with a fiber $X_{t_0}$ being K\"ahler with the K\"ahler form $\omega_{0}$ for some $t_{0}\in S$.
We define
\begin{equation}\label{Sfd}
S_{t_{0}}=\left\{t\in S:\, \text{$X_{t}=(X_{t_{0}})_{\phi(t)}$ for some $\phi(t)\in A^{0,1}\left( X_{t_{0}},\mathrm T^{1,0}X_{t_{0}}\right)$}\right\}.
\end{equation}

Here we write $X_{t}=(X_{t_{0}})_{\phi(t)}$ for some $\phi(t)\in A^{0,1}\left(X_{t_{0}},\mathrm T^{1,0}X_{t_{0}}\right)$ to mean that there exists a diffeomorphism $d_{t}:\, X_{t}\to X_{t_{0}}$ such that $(d_{t}^{-1})^{*}({\mathrm T^{*}}^{1,0}X_{t})$ is locally described by \eqref{basis of cotang}.

According to Definition~4.2 in \cite{Kir}, the points of $S_{t_0}$ may be regarded as being of finite distance from $X_{t_0}$. Motivated by this, one may define the distance between $t_0$ and $t$ in terms of the norm of the corresponding Beltrami differential.

However, for a given complex structure on $X_t$, the choice of a Beltrami differential 
$\phi(t)\in A^{0,1}\!\left(X_{t_0},\mathrm T^{1,0}X_{t_0}\right)$ 
such that $X_t=(X_{t_0})_{\phi(t)}$ is not unique. In particular, when $t\neq t_0$, it may happen that $X_t\cong X_{t_0}$, in which case $\phi(t)$ can be chosen to be zero. We therefore define a non-negative function as the infimum over all such Beltrami differentials, as follows.

\begin{definition}\label{Beltrami distance}
For $t\in S_{t_{0}}$, we define a non-negative function $d_{B}(t_{0},t)$ as
$$d_{B}(t_{0},t) =\inf \left\{\|\phi(t)\|:\, \phi(t)\in A^{0,1}\left( X_{t_{0}},\mathrm T^{1,0}X_{t_{0}}\right) \text{ such that }X_{t}=(X_{t_{0}})_{\phi(t)}\right\}\ge 0.$$
Here the supremum operator norm $\|\phi(t)\|=\|\phi(t)\|^{E}$, as given in Equation \eqref{son E}.
\end{definition}

From Lemma \ref{gK diffeom} below, one sees that $d_{B}(t_{0},t)$ is a continuous function on $S_{t_{0}}$.
Then, for some constant $c>0$, we can define
\begin{equation}\label{Sfdc}
S_{t_{0},c}=\text{connected component of $t_{0}$ in}\left\{t\in S_{t_{0}}:\, d_{B}(t_{0},t)<c\right\}.
\end{equation}

The main result of this section is the following theorem.
\begin{theorem}\label{global Kahler}
Let $f:\,\X \to S$ be an analytic family of compact complex manifolds over a connected complex manifold $S$, with a fiber $X_{t_0}\triangleq X$ being K\"ahler. Let the open subset $S_{t_{0},c}\subset S$ be defined by \eqref{Sfdc} for $c>0$.
Then there exists a constant $c_{0}$, which depends only on the Hodge numbers $h^{2,0}(X)=h^{0,2}(X), h^{1,1}(X)$, such that all the fibers $X_{t}$ are K\"ahler for $t\in S_{t_{0},c_{0}}$. Moreover, the $\nabla_{t_0}^{1,1}$-flat extensions $\mathcal{K}^{\nabla_{t_0}^{1,1}}_{t_{0},t}\subset \mathcal{K}_{t}$ in Theorem~\ref{intr strong Kahler cone invariance} exist for all $t\in S_{t_{0},c_{0}}$.
\end{theorem}

\begin{lemma}\label{gK diffeom}
Let the assumptions be as in Theorem~\ref{global Kahler}. Then the following statements hold:

\emph{(1)} For any $t\in S$ there exists a diffeomorphism
$$
d_{t,t_{0}}:\, X_{t}\longrightarrow X_{t_{0}}.
$$
In particular, if $t\in S_{t_0,c}$, the diffeomorphism $d_{t,t_{0}}$ can be chosen so that the corresponding Beltrami differential $\phi(t)$ satisfies $\|\phi(t)\|<c$.

\emph{(2)} For any $t\in S$, there exists an open neighborhood $U_t$ of $t$ and a family of diffeomorphisms
$$
d(s):\, X_{s}\longrightarrow X_{t_{0}},
$$
depending holomorphically on $s\in U_t$. Moreover, if $t\in S_{t_0,c}$, the neighborhood $U_t$ can be chosen so that the Beltrami differentials $\phi(s)$ induced by $d(s)$ depend holomorphically on $s$ and satisfy $\|\phi(s)\|<c$ for all $s\in U_t$.
\end{lemma}
\begin{proof} 
By Lemma 15.1 in \cite{Clemens}, we have that for any $t\in S$, there exists an open neighborhood $U_{t}$ of $t$ and a diffeomorphism $d_{U_{t}}:\, f^{-1}(U_{t})\to X_{t}\times U_{t}$ with the commutative diagram
$$\xymatrix{
f^{-1}(U_{t}) \ar[r]^-{d_{U_{t}}} \ar[d]^-{f}& X_{t}\times U_{t} \ar[d]^-{\mathrm{pr}_{2}}\\
U_{t} \ar[r]^-{=}&U_{t}.
}$$
Moreover the local diffeomorphism $d_{U_{t}}=D\times f$ is induced by a $C^{\infty}$-projection map
$D:\, f^{-1}(U_{t}) \to X_{t}$ whose fibers are complex holomorphic disks meeting $X_{t}$ transversely. 
From \cite{Clemens}, the corresponding Beltrami differential $\phi(s)\in A^{0,1}(X_{t},\Tan X_{t})$ induced by the diffeomorphism $$d_{t}(s):\,=d_{U_{t}}|_{X_{s}}:\, X_{s}\to X_{t}$$ is holomorphic in $s\in U_{t}$.

Now, for any $t\in S$, we choose a curve $\gamma$ connecting $t_{0}$ to $t$ and cover $\gamma$ by open subsets $U_{t_{i}}$, $0\le i\le n$, satisfying the above properties, where $t_{i}\in \gamma$ and $t_{n}=t$. 
For each $0\le i\le n-1$, choose a point $t_{i,i+1}\in U_{t_{i}}\cap U_{t_{i+1}}$.

The diffeomorphisms
$$d_{t_{i}}(t_{i,i+1}):\, X_{t_{i,i+1}}\to X_{t_{i}}, \, d_{t_{i+1}}(t_{i,i+1}):\, X_{t_{i,i+1}}\to X_{t_{i+1}}$$
induce diffeomorphisms between the fibers
$$d_{t_{i+1},t_{i}}:\, X_{t_{i+1}}\to X_{t_{i}}, \, 0\le i\le n-1.$$
Therefore, we obtain the diffeomorphism
$$d_{t,t_{0}}=d_{t_{1},t_{0}}\circ \cdots\circ d_{t_{n},t_{n-1}}:\, X_{t_{n}}=X_{t}\to X_{t_{0}},$$
which proves (1).

For the proof of (2), we define the diffeomorphism
$$d(s)=d_{t,t_{0}}\circ d_{t}(s):\, X_{s}\to X_{t_{0}},\, s\in U_{t},$$
which depends holomorphically on $s$, since $d_{t}(s)$ varies holomorphically with $s$ and $d_{t,t_{0}}$ is fixed for $s\in U_{t}$.
\end{proof}
\begin{remark}
Since $S$ is not necessarily simply connected, the diffeomorphisms $d(s) :\, X \to X(s)$, for $s\in U_{t}$, are not canonical; that is, they depend on the homotopy type of the curves $\gamma$ from $t_{0}$ to $t$.
\end{remark}

Let $\eta_{(0)},\eta_{(1)},\eta_{(2)}$ be the adapted basis of the Hodge decomposition at the base point 
$$H^{2}(X,\C)=H^{2,0}(X)\oplus H^{1,1}(X)\oplus H^{0,2}(X),$$
where $\eta_{(i)}=[\tilde\eta_{(i)}]$ with $\tilde\eta_{(i)}$ the basis of $\mathbb H^{2-i,i}(X)$, $0\le i\le 2$.

We define the quasi-period map
\begin{eqnarray}
\Phi&:&S_{t_{0},1}\to N_{-}\subset\check D \nonumber \\
&&t\mapsto \Phi(t)=\left(
\begin{array}{ccc}
I & \Phi^{0,1}(t) & \Phi^{0,2}(t) \\
O& I& \Phi^{1,2}(t)\\
O&O&I
\end{array}\right), \label{global qpm}
\end{eqnarray}
by
\begin{eqnarray}
\Phi^{(i,j)}(t)\cdot\eta_{(j)}&=& \left[\mathbb H \left(i_{\phi(t)}^{j-i}\left((I+Ti_{\phi(t)})^{-1}\tilde\eta_{(i)}\right)\right)\right], 0\le i<j\le 2,\label{defn global qpm}
\end{eqnarray}
where $\phi(t)\in A^{0,1}\left( X,\mathrm T^{1,0}X\right)$ such that $X_{t}=X_{\phi(t)}$.

\begin{proposition}\label{DPconj 2}
Let the notations and assumptions be as in Theorem \ref{global Kahler}. 
Then the quasi-period map $\Phi$ in \eqref{global qpm} is well-defined.
Furthermore $\Phi$ is holomorphic and satisfies the Griffiths transversality.
\end{proposition}
\begin{proof} 
For any $t\in S_{t_{0},1}$, we have that $d_{B}(t_{0},t)<1$. By the definition of the Beltrami distance in Definition \ref{Beltrami distance}, there exists a Beltrami differential $\phi(t)\in A^{0,1}\left( X,\mathrm T^{1,0}X\right)$ such that $X_{t}=X_{\phi(t)}$ and $\|\phi(t)\|<1$.

From the results in Section \ref{SHB}, we have the sections 
\begin{align*}
\Omega_{(0)}(t)=&\left[\mathbb H \left(e^{i_{\phi(t)}}\left((I+Ti_{\phi(t)})^{-1}\tilde \eta_{(0)}\right)\right)\right]\nonumber \\
=&\eta_{(0)} + \left[\mathbb H\left(i_{\phi(t)} (I+Ti_{\phi(t)})^{-1}\tilde \eta_{(0)} \right)\right]+\frac{1}{2}\left[\mathbb H\left(i_{\phi(t)}^2 (I+Ti_{\phi(t)})^{-1}\tilde \eta_{(0)} \right)\right]; \\
\Omega_{(1)}(t)=&\left[\mathbb H \left(e^{i_{\phi(t)}}\left((I+Ti_{\phi(t)})^{-1}\tilde \eta_{(1)}\right)\right)\right]\nonumber \\
=&\eta_{(1)} + \left[\mathbb H\left(i_{\phi(t)} (I+Ti_{\phi(t)})^{-1}\tilde \eta_{(1)} \right)\right];\\
\Omega_{(2)}(t)=&\left[\mathbb H \left(e^{i_{\phi(t)}}\left((I+Ti_{\phi(t)})^{-1}\tilde \eta_{(2)}\right)\right)\right]\\
=&\eta_{(2)},
\end{align*}
which are linearly independent for each $t$ and each choice of the Beltrami differential $\phi(t)$. 
Then $\Omega_{(0)}(t),\cdots, \Omega_{(p)}(t)$ give a basis of the filtrations
\begin{equation}\label{Hodge filt 2}
F^{2-p}H^{2}(X_{t},\C)=\mathrm{Im}\, \left(H^{2}(F^{2-p}A^{\bullet}(X_{t});d)\to H^{2}(X_{t},\C)\right)
\end{equation}
for $0\le p\le 2$. Hence there exists blocks $\Phi^{0,1}(t)$, $\Phi^{0,2}(t)$, $\Phi^{1,2}(t)$, such that 
\begin{equation}\label{global qpm 2}
\left(
\begin{array}{ccc}
\Omega_{(0)}(t) \\
\Omega_{(1)}(t)\\
\Omega_{(2)}(t)
\end{array}\right)= \left(
\begin{array}{ccc}
I & \Phi^{0,1}(t) & \Phi^{0,2}(t) \\
O& I& \Phi^{1,2}(t)\\
O&O&I
\end{array}\right)\left(
\begin{array}{ccc}
\eta_{(0)} \\
\eta_{(1)}\\
\eta_{(2)}
\end{array}\right).
\end{equation}

We claim that the blocks $\Phi^{0,1}(t)$, $\Phi^{0,2}(t)$, and $\Phi^{1,2}(t)$ are independent of the choice of the Beltrami differential $\phi(t)$. 
Indeed, the Hodge filtration \eqref{Hodge filt 2} depends only on the complex structure of $X_{t}$, and for any other Beltrami differential $\phi'(t)$ such that
$$X_{t}=X_{\phi(t)}=X_{\phi'(t)},$$
the corresponding period matrix $\Phi'(t)\in N_{-}$ satisfies
$${\Phi'(t)}^{-1}\Phi(t)\in N_{-}\cap B=\{Id\},$$
where $B\subset G_{\C}$ is the parabolic subgroup such that $\check D=G_{\C}/B$. 
This implies that the corresponding blocks of $\Phi(t)$ and $\Phi'(t)$ coincide.

Next we prove that the quasi-period map $\Phi$ is holomorphic. 

Let the holomorphic charts $(U;t)\subset S$ be as in the proof of Lemma~\ref{gK diffeom}. Then locally on $X$, the diffeomorphisms $d(t)$ are given by the holomorphic functions
$$z = (z_1, \dots, z_n) \mapsto (w_1(z, t), \dots, w_n(z, t)),$$
which are holomorphic in $t\in U$. According to Chapter~4.1 of~\cite{MorrowKodaira}, the Beltrami differential has the local expression
$$\phi = \sum_{i,j,k} \left( \frac{\partial w_j}{\partial z_i} \right)^{-1}_{ij}(z, t) \frac{\partial w_j}{\partial \bar{z}_k}(z, t) \, d\bar{z}_k \otimes \partial_{z_i},$$
which is holomorphic in $t$. Therefore, the blocks in~\eqref{defn global qpm} are holomorphic in $t$, which implies that the quasi-period map $\Phi$ is holomorphic.

Finally, Griffiths transversality follows by differentiating the sections given in \eqref{global qpm 2}.
\end{proof}

\begin{proposition}\label{DPconj 3}
Let the notations and assumptions be as in Theorem \ref{global Kahler}. Then 
there exists a positive constant $c_{1}\le 1$, which depends only on the Hodge numbers $h^{2,0}(X)=h^{0,2}(X), h^{1,1}(X)$, such that the image of $S_{t_{0},c_{1}}$ under the quasi-period map in~\eqref{global qpm} lies in $D$; that is, the quasi-period map given in Proposition \ref{DPconj 2} restricts to a period map 
$$\Phi: S_{t_{0},c_{1}} \to N_{-}\cap D.$$
\end{proposition}
\begin{proof} 
The Hodge filtration 
$$F^{2}H^{2}(X_{t},\C)\subset F^{1}H^{2}(X_{t},\C)\subset F^{0}H^{2}(X_{t},\C)=H^{2}(X_{t},\C)$$
gives a pure Hodge structure on $H^{2}(X_{t},\C)$ if and only if
$$F^{1}H^{2}(X_{t},\C)\oplus \bar{F^{2}H^{2}(X_{t},\C)}=H^{2}(X_{t},\C).$$
Hence the image $\Phi(t)$ of the quasi-period map given in Proposition \ref{DPconj 2} lies in $N_{-}\cap D$ if and only if
$$\Omega_{(0)}(t), \Omega_{(1)}(t), \bar{\Omega_{(0)}(t)}$$
are linearly independent, which is equivalent to that
\begin{equation}\label{pure Hodge matrix} 
\mathrm{det}\left(
\begin{array}{ccc}
I & \Phi^{0,1}(t) & \Phi^{0,2}(t) \\
O& I& \Phi^{1,2}(t)\\
\bar{\Phi^{0,2}(t)}&\bar{\Phi^{0,1}(t)}&I
\end{array}\right)\neq 0.
\end{equation}
Since the determinant \eqref{pure Hodge matrix} depends continuously on $t$, there exists a constant $c'_{1}>0$, which depends only on the sizes of the blocks of the matrix, such that the determinant \eqref{pure Hodge matrix} is nonzero whenever $$|\Phi^{0,1}(t)|,|\Phi^{0,2}(t)|,|\Phi^{1,2}(t)|<c'_{1}.$$

Note that the sizes of the blocks are determined by the Hodge numbers $h^{2,0}(X)=h^{0,2}(X), h^{1,1}(X)$. Then from Lemma \ref{constant lemma} below, we conclude that there exists a constant $c_{1}>0$, which depends only on the Hodge numbers, such that the determinant \eqref{pure Hodge matrix} is nonzero for all $t\in S_{t_{0},c_{1}}$. Hence $\Phi(t)\in N_{-}\cap D$ for all $t\in S_{t_{0},c_{1}}$.
\end{proof}

\begin{lemma}\label{constant lemma}
Let the notations be as in Proposition \ref{DPconj 2} and \ref{DPconj 3}. Then we have that 
\begin{eqnarray*}
|\Phi^{0,1}(t)|&\le & \frac{\|\phi(t)\|}{1-\|\phi(t)\|}\\
|\Phi^{0,2}(t)|&\le & \frac{\|\phi(t)\|^{2}}{1-\|\phi(t)\|}\\
|\Phi^{1,2}(t)|&\le & \frac{\|\phi(t)\|}{1-\|\phi(t)\|},
\end{eqnarray*}
for any $\phi(t)\in A^{0,1}\left( X,\mathrm T^{1,0}X\right)$ with $\|\phi(t)\|<1$ such that $X_{t}=X_{\phi(t)}$.
\end{lemma}
\begin{proof} 
We only prove the lemma for $\Phi^{0,1}(t)$ and the proof for other blocks is similar.

From the proof of Proposition \ref{DPconj 2}, we have that 
\begin{eqnarray*}
|\Phi^{0,1}(t)|&= &|\left(\mathbb H\left(i_{\phi(t)} (I+Ti_{\phi(t)})^{-1}\tilde \eta_{(0)} \right), \tilde\eta_{(0)}\right)|\\
&\le &\|\mathbb H\left(i_{\phi(t)} (I+Ti_{\phi(t)})^{-1}\tilde \eta_{(0)} \right) \|\\
&\le &\|i_{\phi(t)}\tilde  \eta_{(0)}-i_{\phi(t)}Ti_{\phi(t)}\tilde \eta_{(0)}+\cdots+i_{\phi(t)}(-Ti_{\phi(t)})^{k-1}\tilde \eta_{(0)}+\cdots\|\\
&\le & \|\phi(t)\|+\|T\|\|\phi(t)\|^{2}+\cdots+\|T\|^{k-1}\|\phi(t)\|^{k}+\cdots\\
&= &\frac{\|\phi(t)\|}{1-\|T\|\|\phi(t)\|}\le  \frac{\|\phi(t)\|}{1-\|\phi(t)\|}.
\end{eqnarray*}
Here we take the adapted basis $\tilde\eta_{(i)}$, $i=0,1,2$, to be the orthonormal basis and the norm $\|T\|\le 1$ follows from Proposition 3.5 in \cite{LS26-1}.
\end{proof}

\begin{proof}[Proof of Theorem \ref{global Kahler}]
As in Section \ref{Kahler section}, we know that a real $2$-class $\sigma=\alpha_{(0)}\cdot \eta_{(0)}+[\omega_{0}]+\bar{\alpha_{(0)}}\cdot \eta_{(2)}$ is a $(1,1)$-class on $X_{t}$ for $t\in S_{t_{0},1}$ if and only if 
\begin{equation}\label{B=01 11-0'}
F=\bar{\alpha_{(0)}}- \alpha_{(1)}^{0}\Phi^{(1,2)}(t)-\alpha_{(0)}\left(\Phi^{(0,2)}(t)-\Phi^{(0,1)}(t)\Phi^{(1,2)}(t)\right)=0,
\end{equation}
where $\alpha_{(1)}^{0}$ is a row vector such that $[\omega_{0}]=\alpha_{(1)}^{0}\cdot \eta_{(1)}$. Here the blocks $\Phi^{0,1}(t)$, $\Phi^{0,2}(t)$, and $\Phi^{1,2}(t)$ are given by Proposition \ref{DPconj 2}.

By the Implicit Function Theorem, we have that Equation \eqref{B=01 11-0'} determines a local real analytic function $\alpha_{(0)}(t)$ in $t$ provided that the full Jacobian matrix 
\begin{equation}\label{fJ weight2}
J(t)=\left(
\begin{array}{cc}
\frac{\partial F}{\partial \bar{\alpha_{(0)}}} & \frac{\partial F}{\partial {\alpha_{(0)}}} \\
\frac{\partial \bar F}{\partial \bar{\alpha_{(0)}}}& \frac{\partial \bar F}{\partial {\alpha_{(0)}}}
\end{array}\right)=\left(
\begin{array}{cc}
I & -A(t) \\
-\bar{A(t)}& I
\end{array}\right)
\end{equation}is non-degenerate at $t$,
where 
$$A(t)=\Phi^{(0,2)}(t)-\Phi^{(0,1)}(t)\Phi^{(1,2)}(t).$$
Since $A(t_{0})=O$, there exists a positive constant $c_{2}\le 1$ such that the Jacobian matrix $J(t)$ is non-degenerate for all $t\in S_{t_{0},c_{2}}$. In fact, from Lemma \ref{constant lemma}
we can choose the constant $c_{2}$, which depends only on the Hodge numbers $h^{2,0}(X)=h^{0,2}(X), h^{1,1}(X)$, such that the operator norm $\|A(t)\|<1$ for whenever $\|\phi(t)\|<c_{2}$.

Let $c_{0}=\min\{c_{1},c_{2}\}>0$. We will prove that for any $t\in S_{t_{0},c_{0}}$, there exists a row vector $\alpha_{(0)}(t)$ satisfying \eqref{B=01 11-0'}.

In fact we can define $S'_{t_{0},c_{0}}$ as the subset of $S_{t_{0},c_{0}}$ consisting of such points.
Since $S_{t_{0},c_{0}}$ is connected, we prove the above statement by a standard open--and--closed argument in topology.

\noindent Openness of $S'_{t_{0},c_{0}}$. 

For $t_{1}\in S'_{t_{0},c_{0}}$ with $\alpha_{(0)}(t_{1})$ and $\sigma(t_{1})$ satisfying \eqref{B=01 11-0'}. Since $J(t)$ is non-degenerate for all $t\in S_{t_{0},c_{0}}$, we have a local real analytic function $\alpha_{(0)}(t)$ around $t_{1}$ satisfying \eqref{B=01 11-0'} with initial value $\alpha_{(0)}(t_{1})$. This proves the openness of $S'_{t_{0},c_{0}}$.

\noindent Closedness of $S'_{t_{0},c_{0}}$. 

Let $\{t_{i}\}_{i=1}^{\infty}$ be a sequence of points in $S'_{t_{0},c_{0}}$ with limit $t_{\infty}\in S_{t_{0},c_{0}}$. We need to show that $t_{\infty}\in S'_{t_{0},c_{0}}$. By the definition of $S'_{t_{0},c_{0}}$, there exist row vectors $\alpha_{(0)}(t_{i})$ such that 
$$\bar{\alpha_{(0)}}(t_{i})- \alpha_{(1)}^{0}\Phi^{(1,2)}(t_{i})-\alpha_{(0)}(t_{i})A(t_{i})=0.$$

We claim that $\{\alpha_{(0)}(t_{i})\}_{i}$ is bounded. Otherwise for any $N>0$,
\begin{eqnarray*}
\|\alpha_{(1)}^{0}\Phi^{(1,2)}(t_{i})\|&\ge&  \|\bar{\alpha_{(0)}}(t_{i})\|-\|A(t_{i})\|\|\alpha_{(1)}^{0}(t_{i})\|\\
&\ge & (1-\|A(t_{\infty})\|-\epsilon_{0}) \|\alpha_{(0)}(t_{i})\|>N
\end{eqnarray*}
for $i$ large enough, where $\epsilon_{0}>0$ can be chosen as $1/2(1-\|A(t_{\infty})\|)$. This contradicts the continuity of $\alpha_{(1)}^{0}\Phi^{(1,2)}(t)$ in a neighborhood of $t_{\infty}$.

Hence $\{\alpha_{(0)}(t_{i})\}_{i}$ is bounded and we can choose a subsequence, which is still denoted by $\{t_{i}\}_{i=1}^{\infty}$, such that $\{\alpha_{(0)}(t_{i})\}_{i}$ has a limit $\alpha_{(0)}({\infty})$. By continuity, the row vector $\alpha_{(0)}({\infty})$ satisfies that
$$\bar{\alpha_{(0)}}({\infty})- \alpha_{(1)}^{0}\Phi^{(1,2)}(t_{\infty})-\alpha_{(0)}({\infty})A(t_{\infty})=0.$$
Therefore $t_{\infty} \in S'_{t_{0},c_{0}}$. This proves the closedness of $S'_{t_{0},c_{0}}$

Now we have proved that for any $t\in S_{t_{0},c_{0}}$, there exists a real $2$-class $\sigma(t)=\alpha_{(0)}(t)\cdot \eta_{(0)}+[\omega_{0}]+\bar{\alpha_{(0)}}(t)\cdot \eta_{(2)}$, which is a $(1,1)$-class on $X_{t}$. Note that $\sigma(t)$ has a harmonic representative 
$$\tilde\sigma(t)=\alpha_{(0)}(t)\cdot \tilde\eta_{(0)}+\omega_{0}+\bar{\alpha_{(0)}}(t)\cdot \tilde\eta_{(2)},$$
which is positive definite from Proposition \ref{intr pos def of Kahler}, provided that $c_0<1$. This proves that $(X_{t},\tilde\sigma(t))$ is K\"ahler. 
\end{proof}

\begin{remark}
Classical results on K\"ahler stability rely on the regularity properties of a family of elliptic operators, which hold only for sufficiently small $|t|$, and from which one cannot directly detect positivity. See \cite{MorrowKodaira}, Chapter~4.4, for further details.

By contrast, our method of K\"ahler stability holds on the ball
$\{t\in S_{t_{0}}:\, d_{B}(t_{0},t)<c_{0}\}$
of radius $c_{0}$, where $c_{0}$ is the largest radius such that all points in the corresponding ball satisfy that the matrices in \eqref{pure Hodge matrix} and \eqref{fJ weight2} are non-degenerate. Note that the entries of the matrices in \eqref{pure Hodge matrix} and \eqref{fJ weight2}, as well as the positivity of the extended $(1,1)$-forms on $X_{t}$, are determined solely by the Beltrami differentials.

All of these arguments are independent of the choice of the initial K\"ahler form $\omega_{0}$. This implies that the $\nabla_{t_0}^{1,1}$-flat extensions $\mathcal{K}^{\nabla_{t_0}^{1,1}}_{t_{0},t}\subset \mathcal{K}_{t}$ in Theorem~\ref{intr strong Kahler cone invariance} exist for all $t\in S_{t_{0},c_{0}}$. This, together with Theorem~0.9 of Demailly--Paun in \cite{DP04}, provides positive evidence for their Conjecture~5.1 in \cite{DP04}.
\end{remark}

As an application of the proof of Theorem \ref{global Kahler}, we will prove in \cite{LS26-3} that a compact complex manifold $X_0$ is K\"ahler if it can be approximated by K\"ahler manifolds arising from a family of compact K\"ahler manifolds with trivial canonical bundles. 
This gives a new proof of Siu's theorem \cite{Siu83} and has further applications to degenerations of hyperk\"ahler manifolds, yielding complete solutions to the conjectures of Perego \cite{Per} and Soldatenkov--Verbitsky \cite{SV24}.

%



\section{Algebraic approximations of compact K\"ahler manifolds}\label{algebraic app section}

In this section, we establish a generalization of Green’s criterion for strong algebraic approximation of compact K\"ahler manifolds, formulated in terms of the Beltrami differential. The classical Green criterion corresponds to the first-order term in our formulation.

We continue to use the notations and basic assumptions from Section \ref{Kahler section}. That is, $(X, \omega_0)$ is a compact K\"ahler manifold, and $\phi(t) \in A^{0,1}(X, \mathrm{T}^{1,0}X)$ is defined as in~\eqref{Kuranishi phi}.

We reformulate the Kodaira problem in the following definition.
\begin{definition}
Let $(X,\omega_{0})$ be a compact K\"ahler manifold. We say that $X$ can be approximated by projective manifolds, or that $X$ admits algebraic approximations, provided that there exists a sequence $\{\psi_{n}\}_{n=1}^{\infty}$ of elements in $A^{0,1}(X,\Tan X)$ such that 
\begin{eqnarray*}
&\bar{\partial} \psi_{n} - \frac{1}{2}[\psi_{n}, \psi_{n}] = 0\\
& \displaystyle \lim_{n\to \infty}\psi_{n}=0
\end{eqnarray*}
and the complex manifold $X_{\psi_{n}}$ is projective for $n\ge 1$.
\end{definition}
From Kuranishi's theorem, $X$ can be approximated by projective manifolds if and only if there exists a sequence $\{t_{n}\}_{n=1}^{\infty}$ in $B$ with limit $\lim t_{n}=0$ such that $X_{t_{n}}$ is projective for $n\ge 1$.

\begin{definition}
Let $(X,\omega_{0})$ be a compact K\"ahler manifold. We say that $X$ can be strongly approximated by projective manifolds, or that $X$ admits strong algebraic approximations, provided that the Kuranishi base $B$ is of positive dimension and the subset
$$B^{a}=\{t\in B:\, X_{t} \text{ is projective}\}$$
is dense in the Kuranishi base $B$ when the radius $\epsilon$ of $B$ is sufficiently small.
\end{definition}

From Theorem \ref{Kahler stability}, we know that $(X_{t},\omega_{t})$ is K\"ahler for $t\in B$. Hence by the Kodaira embedding theorem, Theorem 8.2 of \cite{MorrowKodaira}, $X_{t}$ is projective if and only if $[\omega_{t}]\in H^{2}(X,\mathbb Q)$.

Let 
\begin{eqnarray*}
\Phi&:&\Delta \to N_{-}\cap D\\
&&t\mapsto \left(
\begin{array}{ccc}
I & \Phi^{0,1}(t) & \Phi^{0,2}(t) \\
O& I& \Phi^{1,2}(t)\\ 
O&O&I
\end{array}\right)
\end{eqnarray*}
be the quasi-period map as defined in Section \ref{Kahler section}, and let 
\begin{equation}\label{qpm defn'}
\Phi:\, B \to N_{-}\cap D
\end{equation} be the corresponding period map.

Let $$\eta= \{\eta_{(0)}^T, \eta_{(1)}^T, \eta_{(2)}^T\}^T$$
be the adapted basis with respect to the Hodge decomposition $$H^{2}(X,\C)=H^{2,0}(X)\oplus H^{1,1}(X)\oplus H^{0,2}(X).$$

Let $[\omega_{0}]=\sigma_{0}=\alpha_{(1)}^{0}\cdot \eta_{(1)}$.
Then Proposition \ref{key pp-proposition} implies that
\begin{eqnarray}
 \sigma &=& \alpha_{(0)}\cdot \eta_{(0)}+\alpha_{(1)}\cdot \eta_{(1)}+\alpha_{(2)}\cdot \eta_{(2)} \nonumber\\
    &=& \beta_{(0)}(t)\cdot \Omega_{(0)}(t)+\beta_{(1)}(t)\cdot \Omega_{(1)}(t)+\beta_{(2)}(t)\cdot \Omega_{(2)}(t) \label{alpha,beta 11}
\end{eqnarray}
in $H^{2}(X,\mathbb C)$ is a real $(1,1)$-class on $X_{t}$, if and only if $\alpha_{(1)}=\bar{\alpha_{(1)}}$ is real and the following equation holds
\begin{equation}\label{B=01 11}
\bar{\alpha_{(0)}}- \alpha_{(1)}\Phi^{(1,2)}(t)-\alpha_{(0)}\left(\Phi^{(0,2)}(t)-\Phi^{(0,1)}(t)\Phi^{(1,2)}(t)\right)=0.
\end{equation}
Then \eqref{B=01 11} determines a real analytic function $\alpha_{(0)}=\alpha_{(0)}(\alpha_{(1)},t)$ on $H^{1,1}(X,\mathbb R)\times B$ such that $\alpha_{(0)}(\alpha_{(1)}^{0},0)=0$ and
\begin{equation}\label{H11}
H(\alpha_{(1)},t)=\alpha_{(0)}(\alpha_{(1)},t)\cdot \eta_{(0)}+\alpha_{(1)}\cdot \eta_{(1)} +\bar{\alpha_{(0)}(\alpha_{(1)},t)}\cdot \eta_{(2)} \in H^{1,1}(X_{t},\mathbb R).
\end{equation}
Here we take the Kuranishi base $B$ small enough.

The real $(1,1)$-class $H(\alpha_{(1)},t)$ is still positive definite on $X_{t}$ for $t\in B$ and the radius $\epsilon$ of $B$ sufficiently small. Hence $X_{t}$ is projective if $H(\alpha_{(1)},t)\in H^2(X, \mathbb{Q})$.

It is not easy to determine whether $H(\alpha_{(1)}, t)$ lies in $H^2(X, \mathbb{Q})$ based on the coefficients of $\eta_{(i)}$ in~\eqref{H11}, since $\eta$ is merely an adapted basis of the Hodge decomposition at the base point and is not defined over $\mathbb{Q}$. 
However, by the density of $H^2(X, \mathbb{Q})$ in $H^2(X, \mathbb{C})$, a sufficient condition for this is that the Hodge map
\begin{equation}\label{Hodge map 11}
H:\, H^{1,1}(X,\mathbb R)\times B \to H^2(X, \mathbb{C})
\end{equation}
is an open map near $(\alpha_{(1)}^{0},0)$.

\begin{theorem}\label{alg app main theorem}
Let $(X,\omega_{0})$ be a compact K\"ahler manifold. If there exists $\zeta_{0}=[\tilde\zeta_{0}]\in H^{1,1}(X,\mathbb R)$ such that the map
\begin{equation}\label{iphi}
[\mathbb H \left(i_{\phi(\cdot)}\tilde\zeta_{0}\right)]:\,  B\to H^{0,2}(X)
\end{equation}
is an open map when $B$ is sufficiently small, then $X$ can be strongly approximated by projective manifolds.
\end{theorem}
\begin{proof} 
From the discussion before the theorem, we only need to show that the Hodge map \eqref{Hodge map 11} is an open map around when restricted to $U^{1,1}\times B$, where $U^{1,1}$ is a small enough neighborhood of $\omega_{0}$ in $H^{1,1}(X,\mathbb R)$.

Let 
$$Z= \left\{[\tilde\zeta]\in H^{1,1}(X,\mathbb R):\, \text{the map }[\mathbb H \left(i_{\phi(\cdot)}\tilde\zeta\right)]:\,  B\to H^{0,2}(X)\text{ is not open}\right\}.$$
Since the map $[\mathbb H \left(i_{\phi(\cdot)}\tilde\zeta\right)]$ is linear in $\tilde\zeta$ and $\zeta_{0}\notin Z$ from \eqref{iphi}, the subspace $Z$ is an algebraic subset of $H^{1,1}(X,\mathbb R)$ of codimension $\ge 1$. As the K\"ahler cone of $X$ is an open subset of $H^{1,1}(X,\mathbb R)$, we can choose the original K\"ahler form $\omega_{0}\in H^{1,1}(X,\mathbb R)\setminus Z$.

Hence we only need to prove the theorem under the assumption \eqref{iphi} with $\tilde\zeta_{0}=\omega_{0}$.

From the expression of the Hodge map in \eqref{H11}, we only need to show that the analytic function
\begin{equation}\label{analytic open}
\bar{\alpha_{(0)}}(\alpha_{(1)}^{0},\cdot):\, B\to  \C^{h^{0,2}}
\end{equation}
determined by the equations \eqref{B=01 11} is an open map.

From Griffiths transversality for the period map \eqref{qpm defn'}, c.f. Proposition \ref{qpm 2}, we have that 
$$\left(\Phi^{(0,2)}(t)-\Phi^{(0,1)}(t)\Phi^{(1,2)}(t)\right)_{\mu}^{\bullet}=-\Phi^{(0,1)}(t)\left(\Phi^{(1,2)}(t)\right)_{\mu}^{\bullet},$$
which together with
$$\P^{(0,1)}(t),\P^{(1,2)}(t)=O(|t|)$$
implies that $$\Phi^{(0,2)}(t)-\Phi^{(0,1)}(t)\Phi^{(1,2)}(t)=o(\Phi^{(1,2)}(t)), \text{ as }t\to 0.$$
Hence the openness of the real analytic function in \eqref{analytic open} determined by the equations \eqref{B=01 11} is equivalent to the openness of the holomorphic function
$$\bar{\alpha_{(0)}}=\alpha_{(1)}^{0}\Phi^{(1,2)}(t).$$

From the construction of the quasi-period map in Section \ref{Kahler section}, we have that 
$$\Phi^{1,2}(t)\cdot\eta_{(2)}=\left[\mathbb H \left(i_{\phi(t)}\left((I+Ti_{\phi(t)})^{-1}\tilde\eta_{(1)}\right)\right)\right],$$
which implies that 
$$\alpha_{(1)}^{0}\Phi^{1,2}(t)\cdot\eta_{(2)}=\left[\mathbb H \left(i_{\phi(t)}\left((I+Ti_{\phi(t)})^{-1} \alpha_{(1)}^{0}\tilde\eta_{(1)}\right)\right)\right]=\left[\mathbb H \left(i_{\phi(t)}\left((I+Ti_{\phi(t)})^{-1}\omega_{0}\right)\right)\right].$$
Therefore we reduce the proof to the openness of the map 
$$\left[\mathbb H \left(i_{\phi(\cdot)}\left((I+Ti_{\phi(\cdot)})^{-1}\omega_{0}\right)\right)\right]:\, B\to H^{0,2}(X),$$
which is equivalent to the openness of the map in \eqref{iphi} in the case that $\tilde\zeta_{0}=\omega_{0}$, since the operator norm
$$\|(I+Ti_{\phi(t)})^{-1} -I \|\to 0, \text{ as }t\to 0.$$
Thus we finish the proof of the theorem.
\end{proof}
\begin{remark}
The condition~\eqref{iphi} has the geometric interpretation that the variation of generic classes in $H^{1,1}(X, \mathbb{R})$ maps locally surjectively onto the $H^{0,2}(X)$-component of the base point $X$ via the variation of complex structures along $B$.
\end{remark}

Since $\phi(t)=\sum_{i}\theta_{i}t_{i}+\cdots$ with $\{\theta_{i}\}_{i=1}^{N}$ a basis of $\mathbb H^{0,1}(X,\Tan X)$, a sufficient condition for \eqref{iphi} is that there exists linear subspace of $\mathbb H^{0,1}(X,\Tan X)$, say $\C\{\theta_{i}\}_{i=1}^{h^{2,0}}$, such that 
\begin{equation}\label{Green density condition}
\left\{
\begin{aligned}
&i_{(\cdot)}\tilde\zeta_{0}:\, \C\{\theta_{i}\}_{i=1}^{h^{2,0}} \to H^{0,2}(X) \text{ is an isomorphism};\\
&\mathbb H[\phi(t), \phi(t)]\left|\right._{t_{h^{2,0}+1}=\cdots=t_{N}=0} =0.
\end{aligned}\right.
\end{equation}

\begin{theorem}\label{Green density}
Let $(X,\omega_{0})$ be a compact K\"ahler manifold. If there exists $\zeta_{0}=[\tilde\zeta_{0}]\in H^{1,1}(X,\mathbb R)$ such that condition \eqref{Green density condition} holds, then $X$ can be strongly approximated by projective manifolds.
\end{theorem}

\begin{remark}
(1) When the deformation of $X$ is unobstructed, Theorem~\ref{Green density} corresponds to the Green density criterion; see Proposition~5.20 in~\cite{Voisin2}. Buchdahl independently proved this criterion and used it to give a new proof of the algebraic approximation of compact K\"ahler surfaces in~\cite{Buchdahl06} and~\cite{Buchdahl08}, which was originally established by Kodaira in~\cite{Kodaira63}. The Green density criterion has also been successfully applied in~\cite{Cao15} and~\cite{Graf17} to treat broader cases of compact K\"ahler manifolds.

(2) From our construction, we see that the Green density criterion~\eqref{Green density condition} corresponds to the first-order term of~\eqref{iphi}, which, as noted by Voisin in Chapter~5.3.4 of~\cite{Voisin2}, is never satisfied for $(p,p)$-classes with $p \ge 2$. 
Therefore, our criterion—which takes into account all terms in the Beltrami differential $\phi(t)$ when studying general $(p,p)$-classes—is both crucial and foundational. 
We will elaborate on this point in the next section.
\end{remark}

An interesting case where all terms of the Beltrami differential $\phi(t)$ are involved, giving a broader range of examples for strong approximation by projective manifolds, is when $h^{2,0} = 1$.

\begin{theorem}\label{alg app case1}
Let $(X,\omega_{0})$ be a compact K\"ahler manifold such that $H^{2,0}(X)$ is one-dimensional, generated by $\eta=[\tilde \eta]$ with harmonic representative $\tilde \eta$. 
If the period map for $X$ is non-trivial, or equivalently
\begin{equation}\label{alg app case1 condition}
\mathbb{H}(\phi(t) \lrcorner \tilde\eta)\text{ is not identically zero for $t\in B$},
\end{equation} 
then $X$ can be strongly approximated by projective manifolds.
\end{theorem}
\begin{proof} 
By Theorem~\ref{alg app main theorem}, it suffices to prove that there exists $\zeta_{0}=[\tilde\zeta_{0}]\in H^{1,1}(X,\mathbb{R})$ such that $\mathbb{H} \left(i_{\phi(t)}\tilde\zeta_{0}\right)$ is not identically zero. 
This is equivalent to the statement that the block $\Phi^{1,2}(t)$ of the period map
\begin{eqnarray*}
\Phi &:& B \longrightarrow N_{-}\cap D \\
&& t \longmapsto 
\left(
\begin{array}{ccc}
I & \Phi^{0,1}(t) & \Phi^{0,2}(t) \\
O & I & \Phi^{1,2}(t) \\ 
O & O & I
\end{array}
\right)
\end{eqnarray*}
is not identically zero.

From Lemma \ref{pm nontrivial} below, the theorem is reduced to proving that the block $\Phi^{0,1}(t)$ is not identically zero, which follows from \eqref{alg app case1 condition}.
\end{proof}

\begin{lemma}\label{pm nontrivial}
The following statements are equivalent:
\begin{itemize}
\item[(1)] the period map $\Phi: B \to N_{-}\cap D$ is trivial;
\item[(2)] the block $\Phi^{0,1}(t)$ is identically zero for $t\in B$;
\item[(3)] the block $\Phi^{1,2}(t)$ is identically zero for $t\in B$.
\end{itemize}
\end{lemma}
\begin{proof} 
The implications (2)$\Longrightarrow$(1) and (3)$\Longrightarrow$(1) are obvious.
We only need to show that (1) implies (2) and that (1) implies (3)

\noindent (1) $\Longrightarrow$ (2). We assume that the period map $\Phi$ is non-trivial and $\Phi^{(0,1)}(t)$ is identically zero. According to the Griffiths transversality 
\eqref{trans 2}, we have that
\begin{equation}\label{012}
\left(\Phi^{(0,2)}\right)^\bullet_\mu (t)=\left(\Phi^{(0,1)}\right)^\bullet_\mu (t) \Phi^{(1,2)}(t)\equiv 0,\, 1\le \mu \le N.
\end{equation}
Hence $\Phi^{(0,2)}(t)$ is identically zero, which implies that
$$H^{2,0}(X_t)=F^2H^2(X_t,\C)=F^2H^2(X,\C)=H^{2,0}(X).$$
Then we have that $H^{0,2}(X_t) = H^{0,2}(X)$ under complex conjugate, and $H^{1,1}(X_t) = H^{1,1}(X)$ as orthogonal complements with respect the Poincar\'e bilinear form $Q$.
Therefore the period map $\Phi$ is trivial, which is a contradiction.

\noindent (1) $\Longrightarrow$ (3). We assume that the period map $\Phi$ is non-trivial and $\Phi^{(1,2)}(t)$ is identically zero. Then from \eqref{012} we have that $\left(\Phi^{(0,2)}\right)^\bullet_\mu (t)\equiv 0$ for $1\le \mu \le N$, and hence $\Phi^{(0,2)}(t)\equiv 0$. 
Note that $F^1H^2(X,\C)$ is spanned by $\eta_0$ and $\eta_1$, and $F^1H^2(X_t,\C)$ is spanned by
\begin{eqnarray*}
  \Omega_0(t)&=&\eta_0+\Phi^{(0,1)}\cdot \eta_1+\Phi^{(0,2)}\cdot \eta_2 =\eta_0+\Phi^{(0,1)}\cdot \eta_1\\
  \Omega_1(t)&=&\eta_1+\Phi^{(1,2)}\cdot \eta_2=\eta_1.
\end{eqnarray*}
Therefore we have that
$$F^1H^2(X_t,\C)=F^1H^2(X,\C),$$
i.e.
$$H^{2,0}(X_t)\oplus H^{1,1}(X_t)=H^{2,0}(X)\oplus H^{1,1}(X).$$
We then have that $H^{0,2}(X_t)=H^{0,2}(X)$ by taking orthogonal complements, and $H^{2,0}(X_t)=H^{2,0}(X)$ under complex conjugate. Finally $H^{1,1}(X_t) = H^{1,1}(X)$ as orthogonal complements. Therefore the period map $\Phi$ is trivial, which is a contradiction.
\end{proof}

An easy-to-check condition for~\eqref{alg app case1 condition} is given by the following corollary.

\begin{corollary}\label{alg app case1'}
Let $(X,\omega_{0})$ be a compact K\"ahler manifold such that $H^{2,0}(X)$ is one-dimensional with a generator $\eta$. 
If there exists $\theta\in H^{1}(X,\Theta_X)$ such that $\theta \lrcorner \eta \neq 0$,
then $X$ can be strongly approximated by projective manifolds.
\end{corollary}


\section{Approximation of \texorpdfstring{$(p,p)$}{(p,p)}-classes by Hodge classes}\label{app Hodge section}
In this section, we establish a criterion for the strong approximation of real $(p,p)$-classes on compact K\"ahler manifolds. Given a real cohomology class with harmonic representative, we show that if the associated map to $\bigoplus_{k=1}^p H^{p-k,p+k}(X)$, obtained via successive contractions with the Beltrami differential, is locally open on the deformation space, then the class can be strongly approximated by nearby Hodge classes. This analytic description, involving explicit higher-order terms in the Beltrami differentials, provides a concrete tool for studying the variation of Hodge classes.

Let $(X, \omega_0)$ be a compact K\"ahler manifold with a Beltrami differential
$$
\phi(t) = \sum_{1 \le i \le N} \theta_i t_i + \phi_2(t) + \cdots \in A^{0,1}(X, \mathrm{T}^{1,0}X),
$$
as defined in Section~\ref{Kahler section}, where $\{ \theta_i \}_{i=1}^{N}$ is a basis of $\mathbb{H}^{0,1}(X, \mathrm{T}^{1,0}X)$.

Let
$$
B = \left\{ t \in \Delta \subset \mathbb{C}^N \mid \mathbb{H}[\phi(t), \phi(t)] = 0 \right\}
$$
be the Kuranishi space, which parametrizes the complex structures $\{ X_t \mid t \in B \}$ in a versal way, as in Theorem~\ref{Kuranishi}.

In this section, we generalize the results of Section~\ref{algebraic app section} to $(p,p)$-classes on $X$, and study the conditions under which there exist (strong) approximations of a given real $(p,p)$-class $\sigma_0$ by Hodge classes on nearby complex manifolds $X_t$, $t\in B$. 

Here, a Hodge class $\sigma$ on a complex manifold $M$ is defined to be a rational cohomology class $\sigma \in H^{2p}(M, \mathbb{Q})$ whose image in $H^{2p}(M, \mathbb{C})$ lies in $H^{p,p}(M)$. Let $\mathrm{Hdg}^{2p}(M)$ denote the subspace of $H^{2p}(M, \mathbb{Q})$ consisting of Hodge classes on $M$.

\begin{definition}\label{defn of app by Hodge}
Let $X$ be a compact K\"ahler manifold, and let $\sigma_0 \in H^{p,p}(X,\mathbb R)$ be a real cohomology class. We say that $\sigma_0$ can be approximated by nearby Hodge classes if there exists a sequence $\{t_n\}_{n=1}^{\infty} \subset B$ with $\lim_{n \to \infty} t_n = 0$, and Hodge classes $\sigma_{t_n} \in \mathrm{Hdg}^{2p}(X_{t_n})$ such that
$$
\lim_{n \to \infty} \sigma_{t_n} = \sigma_0.
$$

We say that $\sigma_0$ can be strongly approximated by nearby Hodge classes if there exists an open neighborhood $U^{p,p} \subset H^{p,p}(X,\mathbb R)$ of $\sigma_0$ such that the subset
$$
\left\{ (\zeta, t) \in U^{p,p} \times B \mid \exists\, \sigma_t \in \mathrm{Hdg}^{2p}(X_t) \text{ with } \operatorname{pr}^{p,p}(\sigma_t) = \zeta \right\}
$$
is dense in $U^{p,p} \times B$ for sufficiently small radius $\epsilon$ of the base $B$. Here, $\operatorname{pr}^{p,p}$ denotes the projection map $H^{2p}(X, \mathbb{C}) \to H^{p,p}(X, \mathbb{C})$ induced from the Hodge decomposition at the central fiber $X$.
\end{definition}

\begin{remark}
The strong approximability of $\sigma_0$ can also be characterized in terms of Hodge bundles, as in Proposition~5.20 of~\cite{Voisin2}, when the deformation is unobstructed $B=\Delta$. We summarize this as follows.

Let 
$$
\cdots \subset \mathcal{F}^{k} \subset \mathcal{F}^{k-1} \subset \cdots \subset \mathcal{H}^{2p}:= H^{2p}(X, \mathbb{C}) \times \Delta
$$ 
be the decreasing filtration of Hodge subbundles over $\Delta$. Define
$$
\mathcal{H}_{\mathbb{R}}^{p,p} := \mathcal{F}^p \cap \overline{\mathcal{F}^p} \cap \mathcal{H}_{\mathbb{R}}^{2p}
$$
to be the $C^{\infty}$ subbundle of $\mathcal{H}_{\mathbb{R}}^{2p}$, where
$$
\mathcal{H}_{\mathbb{R}}^{2p} := H^{2p}(X, \mathbb{R}) \times \Delta \subset \mathcal{H}^{2p} 
$$
is the real subbundle of the trivial bundle $\mathcal{H}^{2p}$.

Then the strong approximability of $\sigma_0$ is equivalent to the following condition: the set of rational classes in $\mathcal{H}^{2p}$ is open and dense near $(\sigma_0, 0) \in \mathcal{H}_{\mathbb{R}}^{p,p}$. Here we consider the Hodge bundles as manifolds.
\end{remark}

The main theorem of this section is as follows.

\begin{theorem}\label{app by Hodge main theorem}
Let $X$ be a compact K\"ahler manifold. 
If $\sigma_0=[\tilde \sigma_0] \in H^{p,p}(X,\mathbb R)$ is a real cohomology class  with harmonic representative $\tilde \sigma_0\in \mathbb H^{p,p}(X)$ such that the map
\begin{equation}\label{iphis}
\left(\left[\mathbb H \left(i_{\phi(\cdot)}\tilde \sigma_0\right)\right],\cdots,\left[\mathbb H \left(i^{p}_{\phi(\cdot)}\tilde \sigma_0\right)\right]\right):\,  B\to H^{p-1,p+1}(X)\oplus \cdots \oplus H^{0,2p}(X)
\end{equation}
is an open map when $B$ is sufficiently small, then $\sigma_0$ can be strongly approximated by nearby Hodge classes.
\end{theorem}

Observe that the openness condition in~\eqref{iphis} is linear in $\sigma_0$, and hence holds on a Zariski open subset of $H^{p,p}(X,\mathbb R)$. Hence any real cohomology class $\sigma'_0\in H^{p,p}(X,\mathbb R)$ can be approximated by $\sigma'_i\in H^{p,p}(X,\mathbb R)$ satisfying conditions in~\eqref{iphis}. 

From Theorem \ref{app by Hodge main theorem}, there exists a sequence $\{\sigma'_{ij}\}_{j=1}^{\infty}$ of nearby Hodge classes, i.e. $\sigma'_{ij} \in \mathrm{Hdg}^{2p}(X_{t_{ij}})$ with limits
$$\lim_{j\to \infty}\sigma'_{ij}=\sigma'_i,\, \lim_{j\to \infty}t_{ij}=0.$$
Then we take $\{\sigma'_{ii}\}_{i=1}^{\infty}$ and conclude that $\sigma'_0$ can be approximated by the nearby Hodge classes $\sigma'_{ii}\in \mathrm{Hdg}^{2p}(X_{t_{ii}})$.

\begin{corollary}
Assuming the condition \eqref{iphis} in Theorem \ref{app by Hodge main theorem}, we have that any real cohomology class in $H^{p,p}(X,\mathbb R)$ can be approximated by nearby Hodge classes.
\end{corollary}

In order to prove Theorem \ref{app by Hodge main theorem}, we need to define the quasi-period map of weight $2p$. A straightforward observation shows that the following construction applies directly to arbitrary weights. 

Under the adapted basis with harmonic representatives
$$\eta= \{\eta_{(0)}^T, \cdots, \eta_{(2p)}^T\}^T=\{[\tilde\eta_{(0)}]^T, \cdots, [\tilde\eta_{(2p)}]^T\}^T$$
with respect to the Hodge decomposition $$H^{2p}(X,\C)=\bigoplus_{k+l=2p}H^{k,l}(X),$$ 
the quasi-period map
\begin{equation}\label{qmp 2p form}
\Phi:\, \Delta \to N_{-},\, t\mapsto \Phi(t)
\end{equation}
is given by
$$\Phi(t)=\left(
\begin{array}{cccc}
I & \Phi^{0,1}(t) & \cdots& \Phi^{0,2p}(t) \\
O & I& \cdots& \Phi^{1,2p}(t) \\
\vdots &  \vdots& \ddots&\vdots \\
O&O&\cdots&I
\end{array}\right)$$
such that 
\begin{eqnarray}
\Omega_{(i)}(t)&=&\eta_{(i)}+\sum_{j=i+1}^{2p}\Phi^{(i,j)}(t)\cdot \eta_{(j)} \nonumber \\
&=&\left[\mathbb H \left(e^{i_{\phi(t)}}\left((I+Ti_{\phi(t)})^{-1}\tilde\eta_{(i)}\right)\right)\right],\, 0\le i\le 2p. \label{defn of quasi Omega}
\end{eqnarray}
By comparing types, we have that
\begin{equation}\label{defn of quasi Omega'}
\Phi^{(i,j)}(t)\cdot \eta_{(j)}=\left[\mathbb H \left(i_{\phi(t)}^{j-i}\left((I+Ti_{\phi(t)})^{-1}\tilde\eta_{(i)}\right)\right)\right].
\end{equation}

By the same arguments as in Section~\ref{Kahler section}, we obtain the following results on the properties of our quasi-period maps.

(i). The quasi-period map \eqref{qmp 2p form} is holomorphic in $t\in \Delta$, and takes values in $D$ for the radius $\epsilon$ of $\Delta$ sufficiently small, where $D$ denotes the classifying space of weight $2p$ Hodge structures on $H^{2p}(X, \mathbb{C})$.

(ii). When $t\in B$, $\phi(t)$ represents an integrable complex structure on $X$, which defines a filtration $F^{\bullet}$ on $A^{2p}(X)$ such that
$$e^{i_\phi(t)}: F^kA^{2p}(X)\rightarrow F^k A^{2p}(X_{t}),$$
is an isomorphism for $0\le k\le 2p$. Moreover, from the results in Section \ref{SHB}, we have that 
\begin{equation}\label{section weight 2p}
e^{i_{{\phi(t)}}}\left((I+Ti_{\phi(t)})^{-1}\tilde\eta_{(i)}\right)\in F^{2p-i} A^{2p}(X_{t})
\end{equation}
is $d$-closed, and hence its cohomological classes $[\cdots]=[\mathbb H\cdots]$ lies in
$F^{2p-i} H^{2p}(X_{t},\C)$, for $0\le i\le 2p$. Thus the quasi-period map \eqref{qmp 2p form} restricts to a well-defined period map
\begin{equation}\label{mp 2p form}
\Phi:\, B\to N_{-}\cap D.
\end{equation} 

(iii). The period map \eqref{mp 2p form} satisfies the Griffiths transversality 
\begin{equation}\label{trans 2p}
\left(\Phi^{(i,j)}\right)_{\mu}^{\bullet}(t)= \left(\Phi^{(i,i+1)}\right)_{\mu}^{\bullet}(t)\Phi^{(i+1, j)}(t),\, 0\le i<j\le 2p,
\end{equation}
for the tangent vector $(\cdot)_{\mu}^{\bullet}=\frac{\partial}{\partial t_{\mu}}$ in the Zariski tangent space $T_{t}B$ at $t$.
Hence we have the estimates
\begin{equation}\label{order estimates}
\P^{(i, j)}(t)=O(|t|^{j-i}), \text{ for } 0\le i<j\le 2p,\,t\in B
\end{equation}
of the blocks $\Phi^{(i,j)}(t)$ as in Lemma \ref{order0}.

With the above preparations, we are now in a position to prove the main theorem of this section.
\begin{proof}[Proof of Theorem \ref{app by Hodge main theorem}]
Applying the proof of Theorem \ref{real approximation} to the period map $\Phi:\, B\to N_{-}\cap D$ on the analytic space $B$, we have the Hodge map
\begin{eqnarray}
H&:&H^{p,p}(X,\mathbb R) \times B\to \bigcup_{t\in \Delta_{\epsilon'}}H^{p,p}(X_t,\mathbb R), \label{defn of H B}\nonumber \\
&&(\alpha_{(p)}\cdot \eta_{(p)},t)\mapsto \sum_{0\le i\le p-1}\alpha_{(i)}(\alpha_{(p)},t)\cdot \eta_{(i)}+\alpha_{(p)}\cdot \eta_{(p)} +\sum_{p+1\le j\le 2p}\bar{\alpha_{(j)}(\alpha_{(p)},t)}\cdot \eta_{(j)}.\nonumber
\end{eqnarray}
Here the real analytic functions $\alpha_{(i)}(\alpha_{(p)},t)$, $0\le i\le p-1$, are determined by the equations
\begin{equation}\label{B=02'} 
\left\{
\begin{aligned} 
\bar{\alpha_{(p)}} &=\alpha_{(p)}\\
 F_{1}:&=\bar{\alpha_{(p-1)}} -\sum_{i=0}^{p} \beta_{(i)}(t)\Phi^{(i,p+1)}(t)=0\\
 F_{2}:&= \bar{\alpha_{(p-2)}}-\sum_{i=0}^{p} \beta_{(i)}(t)\Phi^{(i,p+2)}(t) \\
   & \cdots \cdots \\
  F_{p}:&=\bar{\alpha_{(0)}}-\sum_{i=0}^{p} \beta_{(i)}(t)\Phi^{(i,2p)}(t),
\end{aligned}\right.
\end{equation}
and 
\begin{equation}\label{Hodge class B=01'} 
\left\{
\begin{aligned} 
 \beta_{(0)}(t)&=\alpha_{(0)}  \\
\beta_{(1)}(t)&=\alpha_{(1)}-\beta_{(0)}(t) \Phi^{(0,1)}(t)  \\
& \cdots \cdots \\
\beta_{(p)}(t)&=\alpha_{(p)} -\sum_{i=0}^{p-1}\beta_{(i)}(t) \Phi^{(i,p)}(t).\\
\end{aligned}\right.
\end{equation}

Therefore, to prove the theorem it suffices to show that the Hodge map is an open map near $(\alpha_{(p)}^{0},0)$. This is equivalent to saying that the map
\begin{equation}\label{B to barFp+1}
B\to H^{p-1,p+1}(X)\oplus \cdots \oplus H^{0,2p}(X),\, t\mapsto (\bar{\alpha_{(p-1)}},\cdots, \bar{\alpha_{(0)}})
\end{equation}
given by \eqref{B=02'} is locally open, since
$$\frac{\partial F_{j}}{\partial {\alpha_{(i)}}}\bigg|_{t=0}=O,\, 0\le i\le p-1,1\le j\le p.$$

Moreover, from the estimates \eqref{order estimates} and its proof, we have that 
$$\beta_{(0)}(t)\equiv \alpha_{(0)}\, (\mathrm{mod}\, t),$$
and
\begin{equation}\label{B=03} 
\left\{
\begin{aligned} 
 F_{1}&\equiv\bar{\alpha_{(p-1)}} -\alpha_{(p)}\Phi^{(p,p+1)}(t)\, \left(\mathrm{mod}\, t\Phi^{(p,p+1)}(t)\right)\\
 F_{2}&\equiv \bar{\alpha_{(p-2)}}-\alpha_{(p)}\Phi^{(p,p+2)}(t)\, \left(\mathrm{mod}\, t\Phi^{(p,p+2)}(t)\right)\\
   & \cdots \cdots \\
  F_{p}&\equiv\bar{\alpha_{(0)}}-\alpha_{(p)}\Phi^{(p,2p)}(t)\, \left(\mathrm{mod}\, t\Phi^{(p,2p)}(t)\right).
\end{aligned}\right.
\end{equation}
Hence the local openness of \eqref{B to barFp+1} is equivalent to the local openness of the map
\begin{eqnarray*}
&&B\to H^{p-1,p+1}(X)\oplus \cdots \oplus H^{0,2p}(X)\\
&&t\mapsto \left(\alpha_{(p)}\Phi^{(p,p+1)}(t),\cdots, \alpha_{(p)}\Phi^{(p,2p)}(t)\right)=\alpha_{(p)}\left(\Omega_{(p)}(t)-\eta_{(p)}\right)
\end{eqnarray*}
From \eqref{defn of quasi Omega}, we have that 
\begin{eqnarray*}
\alpha_{(p)}\left(\Omega_{(p)}(t)-\eta_{(p)}\right)&=&\alpha_{(p)}\left(\left[\mathbb H \left(e^{i_{\phi(t)}}\left((I+Ti_{\phi(t)})^{-1}\tilde\eta_{(p)}\right)\right)\right]-\eta_{(p)}\right)\\
&\approx&\alpha^{0}_{(p)}\left(\left[\mathbb H \left(e^{i_{\phi(t)}}\tilde\eta_{(p)}\right)\right]-\eta_{(p)}\right)\\
&=&\left[\mathbb H \left(i_{\phi(t)}\tilde \sigma_0\right)\right]+\cdots+\left[\mathbb H \left(i^{p}_{\phi(t)}\tilde \sigma_0\right)\right],
\end{eqnarray*}
where $\approx$ means that the both sides remain close to each other near $(\alpha_{(p)}^{0},0)$.

Now we reduce the local openness of the Hodge map to the local openness of \eqref{iphis}, which finishes the proof of the theorem.
\end{proof}

\begin{remark}
We will explain why the criterion \eqref{iphis} based on the Jacobians involving only first-order derivatives is never satisfied when $p \ge 2$, as noted by Voisin in Chapter~5.3.4 of~\cite{Voisin2}.
 
We take $p=2$ as an example.
From \eqref{B=03}, we get the Jacobians at $t=0$:
\begin{eqnarray*}\label{}
\frac{\partial \bar{\alpha_{(1)}}}{\partial t_{\mu}}\bigg|_{t=0}&=&\left(\alpha_{(2)}\Phi^{(2,3)}\right)_{\mu}^{\bullet}(0) =\alpha_{(2)}\left(\Phi^{(2,3)}\right)_{\mu}^{\bullet}(0),\\
\frac{\partial \bar{\alpha_{(0)}}}{\partial t_{\mu}}\bigg|_{t=0}&=&\left(\alpha_{(2)}\Phi^{(2,4)}\right)_{\mu}^{\bullet}(0) =\alpha_{(2)}\left(\Phi^{(2,4)}\right)_{\mu}^{\bullet}(0)=0,
\end{eqnarray*}
since $\Phi^{(2,4)}(t)=O(|t|^{2})$.

But the openness criterion of the Hodge map, formulated in terms of the Jacobians, is equivalent to requiring that both $\frac{\partial \bar{\alpha}_{(1)}}{\partial t}\big|_{t=0}$ and $\frac{\partial \bar{\alpha}_{(0)}}{\partial t}\big|_{t=0}$ have full rank, which is ruled out by the preceding calculations.

As in the case of the map $z \mapsto z^2$, the Hodge map may still be locally open even when its Jacobian fails to be of full rank.
\end{remark}

Taking into account the first non-trivial terms in the expansion of each $\left[\mathbb H \left(i^{k}_{\phi(\cdot)}\tilde \sigma_0\right)\right]$ for $1 \le k \le p$ in \eqref{iphis}, we derive the following theorem under a weaker hypothesis.

\begin{theorem}\label{app by Hodge case1}
Let $X$ be a compact K\"ahler manifold. If there exists a real cohomology class $\sigma_0\in  H^{p,p}(X,\mathbb R)$ such that the map
\begin{equation}\label{iphis'}
\left((\cdot)\lrcorner \sigma_{0},\cdots,\underbrace{(\cdot)\lrcorner \cdots (\cdot)\lrcorner}_{p}\sigma_{0}\right):\,  H^{1}(X,\Theta_X)\to H^{p-1,p+1}(X)\oplus \cdots \oplus H^{0,2p}(X)
\end{equation}
is surjective, then the real cohomology class $\sigma_0$ can be strongly approximated by nearby Hodge classes. Moreover, any real cohomology class $\sigma'_0 \in H^{p,p}(X,\mathbb R)$ can be approximated by nearby Hodge classes.
\end{theorem}

\section{Hodge locus and variational Hodge conjecture}\label{Hodge locus Hodge conj}
In this section, we give an intrinsic analytic description of the Hodge locus for a compact K\"ahler manifold, expressed explicitly in terms of the Beltrami differential and the harmonic projection. This formulation incorporates higher-order terms and provides a precise analytic subset of the deformation space. We further establish a criterion for the variational Hodge conjecture in the setting of smooth analytic subvarieties, characterizing when the Hodge locus of a subvariety coincides with its deformation locus via the behavior of the Beltrami differential along the normal bundle.

In the previous sections, we consider arbitrary real $(p,p)$-classes and study when they can be approximated by Hodge classes. In this section, we start from a Hodge class
$\sigma \in H^{2p}(X,\mathbb Q)\cap H^{p,p}(X)$
on a compact K\"ahler manifold $X$, and study the space of nearby complex structures on $X$ on which $\sigma$ remain a Hodge class.

Recall that the nearby complex structures are parametrized by the 
Kuranishi space
$$B = \left\{ t \in \Delta \subset \mathbb{C}^N \mid \mathbb{H}[\phi(t), \phi(t)] = 0 \right\},$$
defined by the Beltrami differential
\begin{equation}\label{Beltrami in Hodge locus}
\phi(t) = \sum_{1 \le i \le N} \theta_i t_i + \phi_2(t) + \cdots \in A^{0,1}(X, \mathrm{T}^{1,0}X),
\end{equation}
where $\{ \theta_i \}_{i=1}^{N}$ is a basis of $\mathbb{H}^{0,1}(X, \Tan X)$.

\begin{definition}
Let $X$ be a compact K\"ahler manifold with the Kuranishi space $B$.
For a Hodge class $\sigma \in H^{2p}(X,\mathbb Q)\cap H^{p,p}(X)$, we define the Hodge locus by 
$$B_{\sigma}^{p}:=\{t\in B:\, \sigma \in H^{p,p}(X_{t})\}.$$
\end{definition}

\begin{theorem}\label{Hodge locus main}
Let $X$ be a compact K\"ahler manifold and $\sigma=[\tilde\sigma] \in H^{2p}(X,\mathbb Q)\cap H^{p,p}(X)$ be a Hodge class with the harmonic representative $\tilde\sigma$. Then the Hodge locus $B_{\sigma}^{p}$ can be defined intrinsically by 
\begin{equation}\label{eqn of Hodge locus}
B_{\sigma}^{p}=\left\{t\in \Delta:\, \mathbb H \left({i_{\phi(t)}}\left((I+Ti_{\phi(t)})^{-1}\tilde\sigma\right)\right)=0,\,\mathbb{H}[\phi(t), \phi(t)] = 0\right\},
\end{equation}
which is an analytic subset of $\Delta \subset \mathbb{H}^{0,1}(X, \Tan X)$. 
\end{theorem}
\begin{proof} 
Let  
$$\eta= \{\eta_{(0)}^T, \cdots, \eta_{(2p)}^T\}^T$$
be the adapted basis with respect to the Hodge decomposition on $H^{2p}(X,\C)$. Then we have that $\sigma=\alpha_{(p)}\cdot\eta_{(p)}$.

From the quasi-period map \eqref{qmp 2p form}, defined in Section \ref{app Hodge section}, and the proof of Proposition \ref{key pp-proposition}, we have that $\sigma=\alpha_{(p)}\cdot\eta_{(p)}$
is a Hodge class on $X_{t}$, if and only if 
\begin{equation}\label{B=02''} 
\left\{
\begin{aligned} 
 0&=\alpha_{(p)}\Phi^{(p,p+1)}(t)\\
 0&=\alpha_{(p)}\Phi^{(p,p+2)}(t) \\
   & \cdots \cdots \\
 0&=\alpha_{(p)}\Phi^{(p,2p)}(t).
\end{aligned}\right.
\end{equation}
From Griffiths transversality \eqref{trans 2p}, we have that 
$$\left(\Phi^{(p,p+j)}\right)_{\mu}^{\bullet}(t)=\left(\Phi^{(p,p+1)}\right)_{\mu}^{\bullet}(t)\Phi^{(p+1,p+j)}(t),\, j\ge 2.$$
Hence $\Phi^{(p,p+1)}(t)=0$ in a neighborhood of $t=0$ implies that $\Phi^{(p,p+j)}(t)=0$, $j\ge 2$, in a neighborhood of $t=0$.
Hence \eqref{B=02''} is equivalent to
\begin{equation}\label{B=02'''}
\alpha_{(p)}\Phi^{(p,p+1)}(t)=0.
\end{equation}
From the definition of the quasi-period map in \eqref{defn of quasi Omega'}, we have that 
\begin{eqnarray*}
\alpha_{(p)}\Phi^{(p,p+1)}(t)\cdot \eta_{(p)}&=&\alpha_{(p)}\left[\mathbb H \left({i_{\phi(t)}}\left((I+Ti_{\phi(t)})^{-1}\tilde\eta_{(p)}\right)\right)\right]\\
&=&\left[\mathbb H \left({i_{\phi(t)}}\left((I+Ti_{\phi(t)})^{-1}\tilde\sigma\right)\right)\right],
\end{eqnarray*}
which implies that \eqref{B=02'''} is equivalent to 
\begin{equation}\label{Hodge locus eqn}
\mathbb H \left({i_{\phi(t)}}\left((I+Ti_{\phi(t)})^{-1}\tilde\sigma\right)\right)=0.
\end{equation}
By the above characterization of Hodge classes, we finish the proof of the theorem.
\end{proof}

\begin{remark}
By definition, the Hodge locus $B_{\sigma}^{p}$ is determined by the variation of Hodge structures on $H^{2p}(X_t, \mathbb{C})$, $t\in B$. However, according to our characterization \eqref{eqn of Hodge locus}, the Hodge locus $B_{\sigma}^{p}$ can in fact be described entirely in terms of the Beltrami differential $\phi(t)$ on the central fiber $X = X_0$. Moreover, this characterization is closely related to the variational Hodge conjecture, which we briefly sketch below.
\end{remark}

Let $Z \subset X$ be an analytic subvariety of codimension $p$. Then the cohomology class $\sigma_{Z}$, corresponding to the homology class of $Z$ via the Poincar\'e duality, is a Hodge class of degree $2p$. Let $f:\, \mathcal{X} \to \Delta$ be an analytic family of compact complex manifolds with $X$ as the central fiber, i.e. $X_0 = X$. The variational Hodge conjecture asserts that the Hodge locus $\Delta_{\sigma_{Z}}^{p}$ coincides with the locus $\mathrm{Def}(X,Z)$ of deformations of the pair $(X,Z)$. That is, $\mathrm{Def}(X,Z)$ is an analytic subset of $\Delta$ and there exists an analytic family $\mathcal{Z} \to \mathrm{Def}(X,Z)$ with the commutative diagram
$$\xymatrix{
\mathcal{Z} \ar@{^{(}->}[r] \ar[d]& \mathcal{X}\ar[d]\\
\mathrm{Def}(X,Z) \ar@{^{(}->}[r]& \Delta
}$$
such that for each $t \in \mathrm{Def}(X,Z)$, the fiber $Z_t$ is an analytic subvariety of $X_t$.

In the case where $X$ is not deformation unobstructed, we propose the following reformulation of the variational Hodge conjecture.

\begin{definition}
Let $X$ be a compact K\"ahler manifold, and let $Z \subset X$ be an analytic subvariety of codimension $p$, with associated Hodge class $\sigma_{Z} \in H^{p,p}(X) \cap H^{2p}(X, \mathbb{Q})$. 
We say that the \emph{variational Hodge conjecture holds for $X$ at $Z$} if the Kuranishi base $B$ has dimension at least one, and the Hodge locus $B_{\sigma_{Z}}^{p}$ coincides with the deformation locus $\mathrm{Def}(X,Z) \subset B$ of the pair $(X,Z)$.

We say that the \emph{variational Hodge conjecture holds for $X$} if it holds at every analytic subvariety $Z \subset X$ of every codimension.
\end{definition}

Clearly, $\mathrm{Def}(X,Z)\subset B_{\sigma_{Z}}^{p}$. Thus, the variational Hodge conjecture reduces to showing that $Z$ deforms unobstructedly at every point of $B_{\sigma_{Z}}^{p}$.

Our goal is not to provide a complete criterion for the variational Hodge conjecture for arbitrary analytic subvarieties. Rather, we focus on the case where $Z \subset X$ is a smooth analytic subvariety, that is, a complex submanifold of $X$, and in this setting we obtain a necessary and sufficient criterion for the inclusion
$
B_{\sigma_{Z}}^{p} \subset \mathrm{Def}(X,Z).
$

Let $$N_{Z|X}= \Tan X|_{Z}/\Tan Z$$
denote the normal bundle of $Z$ in $X$, together with the projection map $P_{{Z|X}}:\, \Tan X|_{Z}\to N_{Z|X}$. Let $$\mathbb H_{N_{Z|X}}:\, A^{\bullet,\bullet}(Z,N_{Z|X})\to \mathbb H^{\bullet,\bullet}(Z,N_{Z|X})$$
be the harmonic projection map for the sections of differential forms on $Z$ with values in $N_{Z|X}$. 

From the discussion in Section 11 of \cite{Clemens}, we know that the obstruction of deforming $Z$ in $X$ is 
$$\mathbb H_{N_{Z|X}}\left(\phi(t)|_{N_{Z|X}}\right)\in \mathbb H^{0,1}(Z,N_{Z|X}),$$
where $\phi(t)$ is the Beltrami differential \eqref{Beltrami in Hodge locus} giving the complex structure of $X_{t}$ for $t\in B$, and $\phi(t)|_{N_{Z|X}}=P_{{Z|X}}(\phi(t)|_{Z})$.

\begin{theorem}\label{vhc main}
Let $X$ be a compact K\"ahler manifold with Kuranishi base $B$ of positive dimension, and let $Z \subset X$ be a smooth analytic subvariety. Then the variational Hodge conjecture holds for $X$ at $Z$ if and only if the implication
\begin{equation}\label{vhc criterion}
\mathbb{H}_{N_{Z|X}}\left( \phi(t)\big|_{N_{Z|X}} \right) \neq 0 \;\Longrightarrow\; \mathbb{H}\left(i_{\phi(t)} \tilde{\sigma}_{Z} \right) \neq 0
\end{equation}
holds for every $t \in B$, where $\sigma_{Z} = [\tilde{\sigma}_{Z}]$ denotes the Hodge class associated to $Z$ with harmonic representative $\tilde{\sigma}_{Z}$.
\end{theorem}
\begin{proof} 
We only need to show that $B_{\sigma_{Z}}^{p} \subset \mathrm{Def}(X,Z)$.  
By Theorem~\ref{Hodge locus main} and the discussion in Section~11 of~\cite{Clemens}, this amounts to showing that, for every $t \in B$,
$$\mathbb H \left({i_{\phi(t)}}\left((I+Ti_{\phi(t)})^{-1}\tilde\sigma_{Z}\right)\right)=0\Longrightarrow\; \mathbb{H}_{N_{Z|X}}\left( \phi(t)\big|_{N_{Z|X}} \right)=0,$$
which is equivalent to \eqref{vhc criterion}, since $\mathbb H \left({i_{\phi(t)}}\left((I+Ti_{\phi(t)})^{-1}\tilde\sigma_{Z}\right)\right)$ is close to $\mathbb{H}\left(i_{\phi(t)} \tilde{\sigma}_{Z} \right)$ for $t \in B$ and sufficiently small radius of $B$.
\end{proof}

\begin{remark}
The criterion \eqref{vhc criterion} includes the special case where $\mathbb H_{N_{Z|X}}\left(\phi(t)\big|_{N_{Z|X}}\right)=0$ for every $t\in B$. In this case, the deformation of $Z$ as a submanifold along the Kuranishi family $\mathcal X\to B$ is unobstructed, and hence $\mathrm{Def}(X,Z)=B_{\sigma_Z}^{p}=B$.

Bloch's semi-regularity theorem~\cite{Bloch72} provides a classical sufficient condition for the variational Hodge conjecture. Further examples satisfying the semi-regularity condition have been studied in~\cite{DK16} and~\cite{Kloosterman22}, including smooth projective varieties of dimension $n$ embedded in smooth hypersurfaces of sufficiently high degree in $\mathbb P^{2n+1}$ and certain complete intersections on hypersurfaces of projective space, respectively.

In general, the deformation of $Z$ along the Kuranishi family need not be unobstructed, and the Kuranishi base $B$ may be singular. Our criterion \eqref{vhc criterion} gives a necessary and sufficient condition, expressed explicitly in terms of the Beltrami differential, for $\mathrm{Def}(X,Z)=B_{\sigma_Z}^{p}$. It applies, in particular, to non-semi-regular situations in which $\mathrm{Def}(X,Z)=B_{\sigma_Z}^{p}\subsetneq B$, and thus provides a tool for constructing more examples of the variational Hodge conjecture.
\end{remark}

%




\vspace{+12 pt}

%
%


\begin{thebibliography}{99}




\bibitem{Barlet75}
D.~Barlet,
\newblock{Espace analytique r\'eduit des cycles analytiques complexes compacts d’un espace analytique complexe de dimension finie,}
\newblock{\em Fonctions de Plusieurs Variables Complexes, II (S\'em. Franois Norguet, 1974--1975), Lecture Notes in Math.}, \textbf{482}, Springer-Verlag, New York (1975), pp.~1--158.

\bibitem{Bloch72}
S. Bloch,
\newblock{Semi-regularity and de Rham cohomology,}
\newblock{\em Invent. Math.}, \textbf{17} (1972), pp.~51--66.

\bibitem{Buchdahl06}
N.~Buchdahl,
\newblock{Algebraic deformations of compact K\"ahler surfaces},
\newblock{\em Mathematische Zeitschrift}, \textbf{253} (2006), pp.~453--459.

\bibitem{Buchdahl08}
N.~Buchdahl,
\newblock{Algebraic deformations of compact K\"ahler surfaces II},
\newblock{\em Mathematische Zeitschrift}, \textbf{258} (2008), pp.~493--498.





\bibitem{Cao15}
J.~Cao,
\newblock{On the approximation of K\"ahler manifolds by algebraic varieties,}
\newblock{\em Math. Ann.}, \textbf{363}(1–2) (2015), pp.~393–422.









\bibitem{Clemens} 
H.~Clemens,  
\newblock{ Geometry of formal Kuranishi theory}, 
\newblock{\em Adv. Math.}, \textbf{198} (2005), pp.~311--365.

\bibitem{DK16}
A. Dan and I. Kaur,
\newblock{Semi-regular varieties and variational Hodge conjecture,}
\newblock{\em C. R. Acad. Sci. Paris, Ser. I}, \textbf{354} (2016), pp.~297--300.

\bibitem{Debarre}
O.~Debarre,
\newblock{Periods and moduli,}
\newblock{\em Current Developments in Algebraic Geometry,} MSRI Publications, \textbf{59} (2011), pp.~65--84.




\bibitem{DP04}
J.-P.~Demailly and M.~Paun,
\newblock{Numerical characterization of the K\"ahler cone of a compact K\"ahler manifold,}
\newblock{\em Ann. of Math.}, \textbf{159} (2004), pp.~1247--1274.




\bibitem{Graf17}
P.~Graf,
\newblock{Algebraic approximation of K\"ahler threefolds of Kodaira dimension zero,}
\newblock{\em Math. Ann.}, (2017).









\bibitem{Griffiths1}
P.~Griffiths,
\newblock{Periods of integrals on algebraic manifolds I,}
\newblock{\em Amer. J. Math.}, \textbf{90} (1968), pp.~568--626.

\bibitem{Griffiths2}
P.~Griffiths,
\newblock{Periods of integrals on algebraic manifolds II,}
\newblock{\em Amer. J. Math.}, \textbf{90} (1968), pp.~805--865.

\bibitem{Griffiths69}
P.~Griffiths,
\newblock{On the periods of certain rational integrals: I and II,}
\newblock{\em Annals of Mathematics,} \textbf{90} (1969), pp.~460--495 and 496--541.
















\bibitem{Kir}
G.~Kiremidjian, 
\newblock{Deformations of complex structures on certain noncompact manifolds}, 
\newblock {\em Ann. of Math.}, \textbf{98}, No.~3 (1973), pp.~411--426.

\bibitem{Kloosterman22}
R. Kloosterman,
\newblock{Variational Hodge conjecture for complete intersections on hypersurfaces in projective space,}
\newblock{\em Rend. Semin. Mat. Univ. Padova}, \textbf{148} (2022), pp.~185--201.

\bibitem{Kodaira63}
K.~Kodaira,
\newblock{On compact analytic surfaces. II, III,}
\newblock{\em Ann. of Math. (2)}, \textbf{77} (1963), pp.~563–626; \textbf{78} (1963), pp.~1–40.



\bibitem{KS3}
K.~Kodaira and D.~C.~Spencer,
\newblock {On deformations of complex analytic structures, III,}
\newblock {\em Annals of Mathematics,} Second Series, \textbf{71}(1) (1960), pp.~43--76.






\bibitem{LR}
K.~Liu and S.~Rao, 
\newblock{Remarks on the Cartan formula and its applications},
\newblock {\em Asian J. Math.}, \textbf{16} (2012), pp.~15--169.

\bibitem{LRY}  
K.~Liu, S.~Rao and X.~Yang,
\newblock{Quasi-isometry and deformations of Calabi-Yau manifolds.}
\newblock {\em Invent. Math.}, \textbf{199}, no. 2 (2015), pp.~423--453.



\bibitem{LS26-1}
K.~Liu and Y.~Shen,
\newblock{Sections of Hodge bundles I: Global theory and applications to period maps},
\newblock{\em preprint}, (2026).

\bibitem{LS26-3}
K.~Liu and Y.~Shen,
\newblock{Degenerations and Stability of K\"ahler Structures on Calabi--Yau Manifolds},
\newblock{\em arXiv: 2605.18065}, (2026).

\bibitem{LZ18}
K.~Liu and S.~Zhu,
\newblock{Solving equations with Hodge theory},
\newblock{\em arXiv:1803.01272}, (2018).









\bibitem{MorrowKodaira}
J.~Morrow and K.~Kodaira,
\newblock {\em Complex Manifolds},
\newblock AMS Chelsea Publishing, Porvindence, RI, (2006), Reprint of the 1971 edition with errata.




\bibitem{Per}
A.~Perego,
\newblock{K\"ahlerness of moduli spaces of stable sheaves over non-projective K3 surfaces,}
\newblock {\em Algebraic Geometry} \textbf{6} (2019), 427--453.






\bibitem{RaoWanZhao22}
S.~Rao, X.~Wan, and Q.~Zhao,
\newblock{Power series proofs for local stabilities of K\"ahler and balanced structures with mild $\partial\bar{\partial}$-lemma,}
\newblock {\em Nagoya Mathematical Journal,}
\textbf{246} (2022), pp.~305--354.


%








\bibitem{Siu83}
Y.-T.~Siu,
\newblock{Every K3 Surface is K\"ahler,}
\newblock {\em Inventiones mathematicae,}
\textbf{73} (1983), pp.~139--150.

\bibitem{SV24}
A.~Soldatenkov, M.~Verbitsky,
\newblock{Hermitian-symplectic and K\"ahler structures on degenerate twistor deformations,}
\newblock arXiv:2407.07867 (2024).



%





\bibitem{Voisin2}
C.~Voisin,
\newblock {\em Hodge theory and complex algebraic geometry II},
\newblock Cambridge Universigy Press, New York, (2003).




\end{thebibliography}
\end{document}